\documentclass[12pt,reqno]{amsart}
\normalbaselines	

\usepackage{amsmath}
\usepackage{amsthm}
\usepackage{amssymb}
\usepackage{graphics}
\usepackage{graphicx}
\usepackage{bm}
\usepackage{setspace}

\usepackage{amsfonts}
\usepackage{eucal}
\usepackage{latexsym}
\usepackage{mathdots}
\usepackage{mathrsfs}
\usepackage{enumitem}
\usepackage{stmaryrd,xspace}

\usepackage{color}

\usepackage[%pagebackref,
hypertexnames=false, colorlinks, citecolor=black, linkcolor=blue, urlcolor=red]{hyperref}

\def\nt{\overset{\rotatebox[origin=B]{180}{$\scriptscriptstyle \sphericalangle$}}{\to}  }
\def\bpm{\begin{pmatrix}}
	\def\epm{\end{pmatrix}}
\newcommand{\bee}{\begin{equation}}
\newcommand{\ee}{\end{equation}}
\newcommand{\ba}{\begin{eqnarray}}
\newcommand{\ea}{\end{eqnarray}}
\newcommand{\bi}{\begin{itemize}}
\newcommand{\ei}{\end{itemize}}
\newcommand{\bn}{\begin{enumerate}}
\newcommand{\en}{\end{enumerate}}
\newcommand{\bbm}{\begin{bmatrix}}
\newcommand{\ebm}{\end{bmatrix}}
\newcommand{\bp}{\begin{proof}}
\newcommand{\ep}{\end{proof}}
\newcommand{\nn}{\nonumber}

\newcommand{\mr}{\ensuremath{\mathrm}}
\newcommand{\scr}{\ensuremath{\mathscr}}
\newcommand{\mbf}{\ensuremath{\mathbf}}
\newcommand{\mc}{\ensuremath{\mathcal}}

\newcommand{\ov}{\ensuremath{\overline}}
\newcommand{\sm}{\ensuremath{\setminus}}
\newcommand{\wit}{\ensuremath{\widetilde}}

\newcommand{\fdot}{\,\cdot\,}

\newcommand{\Ga}{\ensuremath{\Gamma}}

\newcommand{\la}{\ensuremath{\lambda }}

\newcommand{\spn}{\operatorname{span}}

\def\C{\mathbb{C}}
\def\R{\mathbb{R}}
\def\D{\mathbb{D}}
\def\T{\mathbb{T}}

\def\N{\mathbb{N}}

\newcommand{\e}{\varepsilon}
\newcommand{\f}{\varphi}

\newcommand{\balpha}{\boldsymbol{\alpha} }
\newcommand{\bmu}{\boldsymbol{\mu}}
\newcommand{\bnu}{\boldsymbol{\nu}}

\newcommand{\be}{\mathbf{e}}
\newcommand{\bff}{\mathbf{f}}
\newcommand{\bh}{\mathbf{h}}
\newcommand{\mathbff}{\mathbf{f}}
\newcommand{\bO}{\mathbf{0}}
\newcommand{\bS}{\mathbf{S}}
\newcommand{\bx}{\mathbf{x}}
\newcommand{\by}{\mathbf{y}}

\newcommand{\dd}{\mathrm{d}}

\newcommand{\wt}{\widetilde}

\newcommand{\clos}{\operatorname{clos}}

\newcommand{\cK}{\mathcal{K}}

\newcommand{\rank}{\operatorname{rank}}
\newcommand{\tr}{\operatorname{tr}}
\newcommand{\im}{\operatorname{Im}}
\newcommand{\re}{\operatorname{Re}}
\newcommand{\CAD}{\operatorname{CAD}}

\newcommand{\cC}{\mathcal{C}}

\newcommand{\cU}{\mathcal{U}}

\newcommand{\1}{\mathbf{1}}

\newcommand{\ip}[2]{\ensuremath{\langle   {#2}  ,  {#1} \rangle}}
\newcommand{\ipcn}[2]{\ensuremath{\left(  {#2},  {#1} \right) _{\mathbb{C} ^n}}}

\newcommand{\ran}{\operatorname{Ran}}

\newcommand{\ci}[1]{_{ {}_{\scriptstyle #1}}}
\newcommand{\ti}[1]{_{\scriptstyle \text{\rm #1}}}
\numberwithin{equation}{section}
\numberwithin{subsection}{section}

\newtheorem{thm}{Theorem}[section]

\newtheorem{lemma}[thm]{Lemma}
\newtheorem{prop}[thm]{Proposition}
\newtheorem{cor}[thm]{Corollary}
\newtheorem{thm*}{Theorem}

\theoremstyle{definition}
\newtheorem{defn}[thm]{Definition}
\theoremstyle{remark}
\newtheorem{remark}[thm]{Remark}
\newtheorem*{remark*}{Remark}

\renewcommand{\labelenumi}{\textup{(\roman{enumi})}}

\newcounter{vremennyj}

\newcommand\cond[1]{\setcounter{vremennyj}{\theenumi}\setcounter{enumi}{#1}\labelenumi\setcounter{enumi}{\thevremennyj}}

\title[Matrix-valued AC measures and CAD]{Matrix-valued  Aleksandrov--Clark measures and Carath\'{e}odory angular derivatives}

\author[C.~Liaw]{Constanze Liaw}
\address{Department of Mathematical Sciences\\
	University of Delaware and\\
	CASPER\\
	Baylor University}
\email{liaw@udel.edu}

\author[R.~T.~W.~Martin]{Robert T.~W.~Martin}
\address{Department of Mathematics\\
	University of Manitoba}
\email{Robert.Martin@umanitoba.ca}

\author[S.~Treil]{Sergei Treil}
\address{Department of Mathematics\\
	Brown Univeristy}
\email{treil@math.brown.edu}

\thanks{The work of C.~Liaw was supported by the National Science Foundation under the grant DMS-1802682. The work of S.~Treil was supported by the National Science Foundation under the grant DMS-1856719.}


\begin{document}

\begin{abstract}
This paper deals with families of matrix-valued Aleksandrov--Clark measures $\{\bmu^\alpha\}_{\alpha\in\cU(n)}$,  corresponding to purely contractive $n\times n$ matrix functions $b$ on the unit disc of the complex plane. We do not make other apriori assumptions on $b$. In particular, $b$ may be non-inner and/or non-extreme. The study of such families is mainly motivated from applications to unitary finite rank perturbation theory.

A description of the absolutely continuous parts of $\bmu^\alpha$ is a rather straightforward generalization of the well-known results for the scalar case ($n=1$).

The results and proofs for the singular parts of matrix-valued $\bmu^\alpha$ are more complicated than in the scalar case, and constitute the main focus of this paper.
We discuss matrix-valued Aronszajn--Donoghue theory concerning the singular parts of the Clark measures,  as well as Carath\'{e}odory angular derivatives of matrix-valued functions and their connections with atoms of $\bmu^\alpha$. These results are far from being straightforward extensions from the scalar case: new phenomena specific to the matrix-valued case appear here. New ideas, including the notion of directionality, are required in statements and proofs.
\end{abstract}

\maketitle

\setcounter{tocdepth}{1}
\tableofcontents

\setcounter{tocdepth}{3}	

\section{Introduction}
In a seminal paper \cite{Clark1972}, D.~Clark  initiated studying families of (scalar, finite, positive and regular) Borel measures $\mu^\alpha$ on the unit circle that correspond to purely contractive analytic functions $b$ on the unit disc $\D$. Namely, for $\alpha\in\T=\partial\D$ the measure $\mu^\alpha$ was defined as the unique measure satisfying 
\begin{align}\label{e-scalarHerglotz}
\frac{\alpha+b(z)}{\alpha-b(z)}
=
i\im{\frac{\alpha+b(0)}{\alpha-b(0)}}
+\int\ci{\partial \D}\frac{\zeta+z}{\zeta-z}\,
\mu^\alpha(\dd \zeta),
\qquad \qquad|\alpha|=1,
\end{align}
(the function in the right hand side is Herglotz, i.e.~it has positive real part, and the above formula is just  the classical Herglotz representation formula).

D.~Clark himself considered the case when $b$ is an inner function, in which case the measures $\mu^\alpha$ are purely singular. In the 1980's and 1990's, A.~B.~Aleksandrov \cite{Aleksandrov, Aleksandrov2, Aleksandrov3, Aleksandrov4, Aleksandrov5} proved many deep results regarding the families of the measures $\mu^\alpha$ (for general, not necessarily inner $b$), which therefore are referred to as Aleksandrov--Clark measures. D.~Sarason \cite{Sarason-dB} explored  the connections between the Clark measures and the corresponding de Branges--Rovnyak spaces. Many deep results about finer properties of the Clark measures were obtained by A.~Poltoratski, \cite{NONTAN, Polt-Sar}.

The Clark measures $\mu^\alpha$ 
are exactly the spectral measures of the unitary rank one extensions of a model operator with the characteristic function $b$.  This was originally shown by Clark \cite{Clark1972} for inner functions $b$; in fact, finding the corresponding spectral measures and investigating their properties was one of the main goals of \cite{Clark1972}.   
For general contractive functions $b$ it was shown significantly later in \cite{LT_JdA} from a different point of view; the measures $\mu^\alpha$ in this case can have non-trivial 
absolutely continuous parts.% 
\footnote{The measures obtained in \cite{LT_JdA}  coincide with the above measures $\mu^\alpha$ if $b(0)=0$; if $b(0)\ne 0$ they differ by a normalizing factor.}

In this paper we are dealing with matrix-valued pure contractions $b$. The analog \eqref{e-Herglotz} of the Herglotz representation formula then defines a family of matrix-valued measures $\bmu^\alpha$ that also has operator theoretic meaning.

We then study the relationship between the properties of the matrix-valued contractions $b$ and their associated  Aleksandrov--Clark family of matrix-valued measures $\bmu^\alpha$. 
As it was mentioned above, in the scalar setting, this topic has been well-developed.  
While there was some development in the matrix-valued case \cite{Polt-Kap_JFA_2006, Elliott2010, Martin-uni, LTFinite}, many fine properties of the matrix-valued Aleksandrov--Clark measures are still not well-understood.

While the characterization of the absolutely continuous part of the matrix-valued Clark measure is pretty simple, capturing the singular part of $\bmu^\alpha$ is more subtle. One of the results of this paper is the description of the \emph{directional support (carrier)} of the singular part of $\bmu^\alpha$; new phenomenon of the directionality appears here. 
In Section \ref{s-Nevanlinna}, we derive an easy Nevanlinna type formula, expressing point masses of $\bmu^\alpha$ in terms of $b$.

In scalar Aleksandrov--Clark theory  
the Aronszajn--Donoghue Theorem \cite{Aronszajn, Donoghue} states that the singular parts of two distinct measures from the same family must be mutually singular.  
Trivially, such a statement cannot be true for the matrix-valued measures $\bmu\ti{s}^\alpha$. However, if one  interprets the mutual singularity as the  \emph{vector mutual singularity} introduced in \cite{LT_JST}, the corresponding result is true, see Corollary \ref{c-orthog}. 
This result is similar (although formally not equivalent) to an earlier result for finite rank perturbations of self-adjoint operators \cite[Theorem 6.2]{LT_JST}. Note that the proof in this paper is also completely different from one in \cite{LT_JST}.

In Section \ref{s-traces}, we use the vector mutual singularity to investigate the ``real'' mutual singularity. 
We show that the exceptional set where the ``real'' mutual singularity fails is small,  see Theorem \ref{t:countable_A-D} below. Again the result is similar to one for finite rank perturbations in \cite[Theorem  6.1]{LT_JST}.

Sections \ref{s-CAD} through \ref{s-CaraPoint} are devoted to extending the notion of Carath\'{e}odory on angular derivative to the matrix-valued setting. The work of Carath\'{e}odory on angular derivative plays an important role in the classical complex analysis; there are deep connections with composition operators, see \cite{Shapiro_book_comp, Cowen-MacCluer_book}, the de Branges--Rovnyak spaces, see \cite{Sarason-dB}, theory of rank one perturbations.

We introduce a directional Carath\'{e}odory condition in Definition \ref{d-CCond}. As in \cite[Chapter VI]{Sarason-dB} and \cite[Section 5.1]{Martin-uni}, this condition can be related to properties of the de Branges--Rovnyak space of $b$, see Proposition \ref{p-equivalences}.
We further introduce the notion of a Carath\'{e}odory angular derivative (CAD) on subspaces (Definition \ref{d-CAD}); note that as the counterexample presented in Section \ref{s:CAD counterexample} shows, a straightforward generalization of the scalar definition does not work well for the matrix case, and a bit more involved definition is needed.   In Theorems \ref{t:CAD}, and \ref{t-P} through \ref{t-E} we find relations between this CAD and the Carath\'{e}odory condition, boundary reproducing kernels for the de Branges--Rovnyak space, and  as the (matrix-valued) point masses of $\bmu^\alpha$.

\section{Preliminaries}\label{s-prelim}
%%%%%%%%%%%%%%%%%%%%%%%%%%%%%%%%%%%%%%%%%%%%%%

\subsection{Matrix-valued Aleksandrov--Clark measures}\label{ss-matrixACmeas}
%%%%%%%%%%%%%%%%%%%%%%%%%%%%%%%%%%%%%%%%%%%%%%
Let $H^\infty (\D ) \otimes \C ^{n\times n} $ denote space of bounded $n\times n$ matrix-valued functions on the open complex unit disc $\D$. In this paper $\C^{n\times n}$ denote the set of all $n\times n$ complex matrices equipped with the operator norm (maximal singular value), i.e.~the set of all (bounded) linear operators on $\C^n$. We define the 
matrix-valued Schur class $\scr{S} (n)$, to be the set of all \emph{purely contractive}  functions in $H^\infty (\D ) \otimes \C ^{n\times n} $. Recall that a function  $b \in H^\infty (\D ) \otimes \C ^{n\times n} $ is purely contractive if and only if $\| b (z) \| <1$ for all $z \in \D$. Note that $\| b(z) \| <1$ if and only if $\| b(0) \| <1$ by the Schwarz lemma (and M\"{o}bius transformations).

Let $\cU(n)$ denote the group of unitary $n\times n$ matrices.  

Given $b \in \scr{S} (n)$ and $\alpha \in \mc{U} (n)$
define the 
function:
\begin{align}
\label{Herglotz fn}
H_{\alpha} (z) := (I_n + b(z) \alpha ^* ) (I_n - b(z) \alpha ^* ) ^{-1}.   
\end{align}
It is easy to see that $H_\alpha$ is a Herglotz function on $\D$, i.e.~an analytic function with non-negative real part on $\D$.
Using the parallelogram identity, it is not difficult to obtain from the classical scalar Herglotz representation formula \eqref{e-scalarHerglotz} its matrix-valued version\footnote{Matrix-valued and operator-valued Herglotz--Riesz representation formulas have been subject to much research as early as \cite{ShmulYan}.}. Namely, for each $\alpha\in\cU(n)$, there is a unique  finite, non-negative $\mathbb{C} ^{n\times n}$-valued Borel measure $\bmu_{\alpha}$ on the unit circle $\T=\partial \D$, so that
\begin{align}\label{e-Herglotz}
H_{\alpha} (z) = i \im{H_{\alpha} (0)} + \int _{\partial \D} \frac{\zeta+z }{\zeta-z} \bmu ^{\alpha} (\dd\zeta ) . 
\end{align}
To avoid misunderstandings, we mention that the imaginary and real part of a matrix $A$ is given by $\im{A}:=(A-A^*)/(2i)$ and $\re{A}:=(A+A^*)/2$ respectively.

The measures $\bmu^{\alpha}$ are called  Clark or Aleksandrov--Clark measures (for $b$). To our knowledge, this definition was first introduced in \cite{Polt-Kap_JFA_2006} for  operator-valued  \emph{inner} functions $b$,  and then in \cite{Elliott2010} for general contractive \emph{matrix-valued} functions.

Note that replacing in \eqref{Herglotz fn} the expression $b(z)\alpha^*$ by $\alpha^* b(z)$ we still get a Herglotz function, so one can wonder why we use this particular order in \eqref{Herglotz fn}. One of the reasons is the theory of the matrix-valued de Branges--Rovnyak spaces.

Recall, see \cite{Martin-uni}, that for $b\in\scr{S}(n)$ the de Branges--Rovnyak space  $\scr{H}(b)$ is the $\C ^n$-valued reproducing kernel Hilbert space (RKHS) with %$\C ^{n\times n}$
matrix-valued reproducing kernel:
\begin{align}\label{d-kernel}
k^b (z,w) := \frac{I_n - b(z) b(w)^*}{1-z\overline w}; \quad \quad z,w \in \D. 
\end{align}
Note that for any $\alpha \in \mc{U} (n)$, $k^{b} = k ^{b\alpha ^*}$ (so that $\scr{H} (b) = \scr{H} (b\alpha ^*)$), but generally $k^\alpha\ne k^{\alpha^*b}$, which motivates our choice of order. 

We also mention that the order $b(z)\alpha^*$ agrees well with the Clark model for finite rank perturbations developed in \cite{LTFinite}.

%%%%%%%%%%%%%%%%%%%%%%%%%%%%%%%%%%
\subsection{Trace and decomposition of a matrix-valued measure}\label{ss-tracedecomp}
We are interested in the subtle properties of the Aleksandrov--Clark family of measures. In order to formulate these precisely, we introduce some terminology.

For matrix-valued measure $\bmu$, define the trace $\mu := \tr \bmu = \sum_{k=1}^n (\bmu)_{k,k}$, where $(\bmu)_{k,l}$, $1\le k,l\le n$ is the $(k,l)$-entry of $\bmu$. Recall that the operator norm of a matrix $A$ is bounded by its trace. Indeed, we have $\|A\|\le \tr ((A^*A)^{1/2})$ and for positive definite matrices $A=(A^*A)^{1/2}.$
In particular, there exists a measurable matrix-valued function $W$ mapping the unit circle $\mathbb{T}$ to positive definite $n\times n$ matrices so that
\[
\dd\bmu (\lambda) = W(\lambda) \dd \mu(\lambda);
\qquad
\lambda\in \T.
\]
Of course, the entries of $W$ are defined a.e.~with respect to $\mu$. This definition of $W$ through the trace also ensures that its entries are in $L^\infty.$ In fact, we have $\tr W(\lambda)\le 1$ with respect to $\mu$-a.e.~$\lambda\in \T.$

Through the Lebesgue decomposition of the scalar measure $\dd\mu = w \dd m+ \dd\mu\ti{s}$ (here, we denote by $m$ the normalized Lebesgue measure on $\T$) we decompose the matrix-valued measure $\bmu$ correspondingly. Concretely, we have
\[
\dd\bmu = W w\dd m  + W \dd\mu\ti{s} =
\widetilde{W} \dd m+ W \dd\mu\ti{s} = \dd\bmu\ti{ac}+\dd\bmu\ti{s},
\]
where $\widetilde W = W w$, $w = \dd\mu / \dd m$ (we can also write $\widetilde W = \dd\bmu / \dd m$).

%%%%%%%%%%%%%%%%%%%%%%%%%%%%%%%%%%%%%%
\subsection{Some known 
	and some simple 
	results on Aleksandrov--Clark measures}
Let $z \nt  \la$ denote non-tangential convergence of $z \in \D$ to $\la \in \partial \D$. Recall that one says that $z\nt\la$ if $z$ approaches $\la$ from within a \emph{Stolz region}:
$$ \Ga _t (\la) := \{ z \in \D : \ |z-\la | < t \left( 1 - |z| \right) \}, \quad \quad t >1. $$ 

It is well-known that for every $t>1$ the non-tangential boundary values of $b$ exist with respect to Lebesgue a.e.~$\la\in \T$. For $\la\in \T$, we let $b(\la):=\lim_{z \nt  \la}b(z)$ wherever the non-tangential limit exists.

The description of the closed support of the scalar measure is an easy and well-known fact (cf., e.g.~\cite[Corollary 4.4]{Mitkovski_IUMJ}). 

The multiplicity of the absolutely continuous part of $\bmu^\alpha$ can be captured in terms of the non-tangential boundary values $b(\la)$ of the characteristic function. To do so, we now state and prove a version of \cite[Theorem 5.6]{LTFinite}.
Recall that $b$ is a contraction on $\mathbb{D}$. For $z\in \mathbb{D}$ define the defect function
\[
\Delta_\alpha(z) := (I_n - \alpha b(z)^*b(z)\alpha^*)^{1/2}.
\]
Through taking non-tangential boundary values of $b$, we can also define $\Delta_\alpha(\la)$ a.e.~on $\T$.
Consider the Lebesgue density $\wt W^\alpha$ of $\bmu^\alpha$. See Subsection \ref{ss-tracedecomp} for the definition.

\begin{thm}
Take $\alpha\in \mc{U}(n)$. The Lebesgue density $\widetilde W^\alpha$  of the Aleksandrov--Clark measure $\bmu^\alpha$  can be computed as
\[
\widetilde W^\alpha(\lambda) 
=
(I_n-\alpha b(\lambda)^*)^{-1}(\Delta_\alpha(\lambda))^2 (I_n-b(\lambda)\alpha^*)^{-1}, \qquad\text{for a.e.~}\lambda\in\T
\]
(note that $I_n-b(\lambda)\alpha^*$ is invertible for a.e.~$\lambda \in\T$).
In particular, its rank is
$$
\rank\,\widetilde W^\alpha(\lambda) = \rank\,\Delta_\alpha(\lambda) \qquad\text{for a.e.~}\lambda\in\T.
$$
\end{thm}

\begin{proof}
	We begin by taking the real part of \eqref{e-Herglotz}. With the Poisson extension $\mathcal{P}(\bmu^\alpha)$ of matrix-valued measure $\bmu^\alpha$, we obtain and then evaluate
	\begin{align*}
	\mathcal{P}(\bmu^\alpha)
	&=
	\re{[(I_n + b(z) \alpha ^* ) (I_n - b(z) \alpha ^* ) ^{-1}]}\\
	&=
	(I_n - \alpha b(z)^*)^{-1}
	(\Delta_\alpha(z))^2
	(I_n - b(z)\alpha^*)^{-1}.
	\end{align*}
	Since $b$ is a strict contraction on $\mathbb{D},$ the inverses exists there.
	
	To obtain the desired result we take $z\nt\la$. To see what happens on the left hand side we recall that, by Fatou's lemma, the non-tangential boundary values of the Poisson extension of a complex measure equal (Lebesgue almost everywhere) to the absolutely continuous part of the measure.  
	On the right hand side, we argue factor-wise. As was discussed before the theorem, the non-tangential boundary values of $\Delta_\alpha(z)$ exist. And since $b$ is a contraction on $\D$, $\det (I_n - b(z)\alpha ^*)$ is a non-trivial analytic function on $\D$. By the uniqueness theorem, it has non-trivial boundary values Lebesgue a.e.~on $\T$ and so $I_n - b(z)\alpha ^*$ is invertible a.e.~on $\T$. To see that $I_n - \alpha b(z)^*$ is invertible a.e.~on $\T$, simply work with the complex conjugate of the anti-analytic function.
\end{proof}

The following standard result (see e.g.~\cite[Theorem 6.1]{Gesztesy2000} for a real-line analog) will enable a recovery (see Corollary \ref{c-TraceSing}) of some information regarding the location of the singular part. We say that a Borel set $X$ is a carrier of a measure, if the measure of $\mathbb{T}\setminus X$ vanishes.

Define the matrix-valued (or scalar-valued) Cauchy transform of a  matrix-valued (or scalar-valued, respectively) measure
\begin{align}\label{d-Cauchy}
\cC\bnu(z)
:=
\int _{\partial \D} \frac{1}{1-\overline\zeta z} \bnu (d\zeta)
\qquad 
\text{for }z\in \mathbb{D}.
\end{align}

\begin{prop}\label{p-SP}
	Let  $\bnu$ be a $\mathbb{C}^{n\times n}$ regular finite positive Borel measure on $\mathbb{T}$.  Consider the sets
	\begin{align*}
	S:=\left\{\lambda\in \mathbb{T}: \lim_{z \nt\lambda}\tr\, \re\, \cC\bnu(z)=  \infty\right\},
	\,\,
	P:=\left\{\lambda\in \mathbb{T}: \lim_{z \nt\lambda}\tr\, (z-\lambda) \cC\bnu(z)\neq 0\right\}.
	\end{align*}
	where $\cC$ is the Cauchy transform given by \eqref{d-Cauchy}.
	
	Then set $S$ is a carrier of the singular part $\bnu\ti{s}$ of $\bnu$ and the set $P$ is the carrier ope the purely atomic part $\bnu\ti{a}$ of $\bnu$. Moreover, $S$ has zero Lebesgue measure, and $P$ is a minimal carrier of $\bnu\ti{a}$, meaning that no proper subset $Y\subsetneq P$ is a carrier of $\bnu\ti{a}$. 
\end{prop}

On the side, we mention that Proposition \ref{p-SP} is an immediate corollary to the analogous result for scalar measures. Indeed, the carrier of a matrix-valued measure is that of its trace, and taking the Cauchy transform, taking the trace and taking the real part all commute.

Moving on, we easily obtain the corresponding result for $\bmu_\alpha.$ 

\begin{cor}\label{c-TraceSing}
Consider the sets
	\begin{align*}
	S_\alpha&=\left\{\lambda\in \mathbb{T}: \lim_{z \nt\lambda}\tr\,\re\, (I_n - b(z)\alpha^*)^{-1}=  \infty\right\},\\
	P_\alpha&=\left\{\lambda\in \mathbb{T}: \lim_{z \nt\lambda}\tr\, (z-\lambda) (I_n - b(z)\alpha^*)^{-1} \neq 0\right\}.
	\end{align*}
	Then the set $S_\alpha$ is a carrier of the singular part $\bmu^\alpha\ti{s}$ of $\bmu^\alpha$ and $P$ is a carrier of the purely atomic part $\bmu^\alpha\ti a$ of $\bmu^\alpha$. Moreover,   $S$ has Lebesgue measure zero, and $P$ is a minimal carrier of $\bmu^\alpha\ti a$, meaning that no subset $Y\subsetneq P_\alpha$ is a carrier  of $\bmu^\alpha\ti a$.
\end{cor}

\begin{proof}
	We use the identity
	$ \frac{1+z\bar\zeta}{1-z\bar\zeta} = 2(1-z\bar\zeta) ^{-1} -1 $
	on the left and right hand side of \eqref{e-Herglotz} to re-arrange the Herglotz formula to read
	\begin{align}\label{e-Cauchy}
	(I_n - b(z) \alpha ^* ) ^{-1} = \frac{I_n - H_\alpha (0) ^*}{2} + \int _{\partial \D} \frac{1}{1-z\bar\zeta} \bmu^{\alpha} (d\zeta ). 
	\end{align}
	Now Proposition \ref{p-SP} (applied to measure $\bmu^\alpha$) immediately yields the result.
\end{proof}

%%%%%%%%%%%%%%%%%%%%%%%%%%%%%%%%%%%%%%%%%%%%%%%%%%%%%%%%%%%%%%
	\subsection{Poltoratski's Theorem}
	The following theorem by Poltoratski, see \cite[Theorem 2.7]{NONTAN}, often plays a key role in investigations of the singular parts of Aleksandrov--Clark measures.
	
	\begin{thm}\label{t-Polt}
		For a (scalar) finite Borel measure $\tau$ on $\T$ and $f\in L^2(\tau)$, the normalized Cauchy transform $\frac{\cC f\tau(z)}{\cC\tau(z)}$ possesses the following non-tangential boundary values 
		$\tau\ti{s}$-a.e.:
		\begin{align*}
		\lim_{z \nt\lambda}
		\frac{\cC f\tau(z)}{\cC\tau(z)}
		= f(\lambda)
		\quad
		\text{for } 
		\tau\ti{s}\text{-a.e.~}\lambda\in\T.
		\end{align*}
	\end{thm}

%%%%%%%%%%%%%%%%%%%%%%%%%%%%%%%%%%%%%%%%%%%%%%%%%%%%%%%%%%%%%%%%%%%%%%%%%%%%%%%%%%%%%%%%%%%%%%%%%%%%
\section{Nevanlinna Theorem concerning point masses}\label{s-Nevanlinna}
We refine the second statement of Corollary \ref{c-TraceSing} in the following simple matrix-valued analog of a result by Nevanlinna.

\begin{thm}\label{t-Nevanlinna}
	Fix $b \in \scr{S} (n)$ and $\alpha \in \mc{U} (n)$. Then for any $\la \in \partial \D$, 
	$$ \bmu^\alpha \{\la\} = \lim _{z \nt \la} (1-z\bar\la) (I_n - b(z) \alpha ^* ) ^{-1}
	$$
	(the limit exists for all $\lambda\in\T$).
\end{thm}

Throughout this paper, we use $\bmu\{\lambda\}$ to denote $\bmu(\{\lambda\})$.

\begin{remark*}
In the scalar situation ($n=1$) the classical Nevanlinna theorem is usually stated as follows: For $\lambda\in\T$ one has  $\mu^\alpha\{\lambda\} \ne 0$ if and only if 
\begin{align}
\label{e: Nevanlinna 01}
\lim_{z\nt\lambda} b(z)= \alpha, \qquad \text{and }\quad b'(\lambda):=\lim_{z\nt\lambda} b'(z)   \quad\text{exists and is finite.}
\end{align}
The limit $b'(\lambda)$ is called the Carath\'{e}odory angular derivative. Note that the conditions \eqref{e: Nevanlinna 01} are equivalent to the existence of the limit
\[
\lim_{z\nt \lambda} \frac{b(z)-\alpha}{z-\lambda} ,
\]
and that this limit coincides with $b'(\lambda)$. As for the $\mu^\alpha\{\lambda\}$, the statement found in the literature usually states that $\mu^\alpha\{\lambda\} = 1/|b'(\lambda)|$; however one can see from the proof (and it was stated in the original Nevanlinna paper \cite{Nevanlinna}) that $\mu^\alpha\{\lambda\} = \alpha \overline \lambda/ b'(\lambda) $. 

As one can easily see, in the scalar case our result gives exactly the same value. However, in the matrix case the relations with the Carath\'{e}odory angular derivative is more complicated (one needs to take into account the directionality of derivatives). The complete theory of the Carath\'{e}dory angular derivatives in the matrix case will be presented below in Sections \ref{s-CAD}, \ref{s-CAD2}, \ref{s-CaraPoint}. 

	The authors thank H.~Woerdeman for asking the question that prompted this remark.
\end{remark*}

\begin{proof}[Proof of Theorem \ref{t-Nevanlinna}]
	Multiplying both sides of equation \eqref{e-Cauchy} by $1-z\bar\lambda$ and taking non-tangential limits, it follows that 
	\begin{align*}
	\lim _{z \nt \la} (1-z\bar\la) (I_n - b(z) \alpha ^* ) ^{-1}  &=  0 + \lim _{z \nt \la} \int _{\partial \D} \frac{1-z\bar\la}{1-z\bar\zeta} \bmu^{\alpha} (d\zeta)  \\
	& =  \bmu^{\alpha} \{\la\} + \lim _{z \nt \la} \int _{\partial \D \sm \{ \la \} } \frac{1-z\bar\la}{1-z\bar\zeta} \bmu^{\alpha} (d\zeta). 
	\end{align*}
	Fix a $t^2 >0$, and a Stolz domain $\Ga _t (\la )$, so that $z \rightarrow \la$ from within $\Ga _t (\la)$. For each such $z$, consider the integrand:
	$$ f_z (\zeta ) := \frac{1-z\bar\la}{1-z\bar\zeta}; \quad \quad \zeta \in \partial \D \sm \{ \la \}. $$ This is a uniformly bounded (in modulus) net (indexed by $z$) of functions on $\partial \D \sm \{ \la \}$ since:
	\[ |f _z (\zeta ) |  =  \left| \frac{1-z\bar\la}{1-z\bar\zeta}  \right| = \left| \frac{z-\la}{z-\zeta} \right| \nn 
	<  t^2 \frac{(1-|z|)}{|z- \zeta|}
	\leq t^2 \frac{(1-|z|)}{|\zeta| - |z|} \nn 
	=  t^2 < \infty. \nn \]
	Moreover, for any fixed $\zeta \in \partial \D \sm \{ \la \}$, 
	$$ \lim _{z \nt \la} f_z (\zeta ) = 0, $$ so that the (moduli of the) $f_z (\zeta)$ are dominated by the constant function $t^2$, and converge to $0$ pointwise on $\partial \D \sm \{\la \}$. By the Lebesgue dominated convergence theorem, 
	$$ \lim _{z \nt \la} \int _{\partial \D \sm \{ \la \} } \frac{1-z\bar\la}{1-z\bar\zeta} \bmu^{\alpha} (d\zeta) =0, $$ and the claim follows.
\end{proof}

%%%%%%%%%%%%%%%%%%%%%%%%%%%%%%%%%%%%%%%%%%%%%%%%%%%%%%%%%%%%%%%%%%%%%%%%%%%%%%%%%%
\section{Directional carrier and vector mutual singularity of singular parts}\label{s-support}
%%%%%%%%%%%%%%%%%%%%%%%%%%%%%%%%%%%%%%%%%%%%%%%%%%%%%%%%%%%%%%%%%%%%%%%%%%%%
We refine Corollary \ref{c-TraceSing} to include a directional carrier of the matrix-valued singular part.

\begin{prop}\label{p-sing}
	For every $\be\in \ran{W^\alpha(\la)}$, the non-tangential limit $\lim_{z\nt\lambda}b^*(z)\be $ exists $\mu^\alpha\ti{s}$-a.e. and is equal to $\alpha^*\be $.
\end{prop}

Ramifications of this proposition are the vector mutual singularity of the matrix-valued measures (Corollary \ref{c-orthog} below), as well as the strong  mutual singularity result  (Theorem \ref{t:countable_A-D}).

\begin{proof}[Proof of Proposition \ref{p-sing}]
We take the adjoint of \eqref{e-Cauchy} and then multiply from the left by $(I_n-\alpha b(z)^*)(\overline{\cC\mu^\alpha(z)})^{-1}$, where the Cauchy transform $\cC\mu^\alpha(z)$ was defined in \eqref{d-Cauchy}.
We arrive at
\[
(I_n-\alpha b(z)^*)\left(\frac{\cC\bmu^\alpha(z)}{\cC\mu^\alpha(z)} \right)^*
=
(\overline{\cC\mu^\alpha(z)})^{-1}
\left[
I_n-I_n/2+H_\alpha(0)/2\right].
\]

Now we take non-tangential limits as $z\nt\lambda$. Recalling that the non-tangential boundary limits of $\cC\mu^\alpha(z) = \infty$ with respect to $\mu^\alpha\ti{s}$-a.e., we obtain
\begin{align}\label{e-tozero}
\lim _{z \nt \la} (I_n -  \alpha b(z)^* ) \left(\frac{\cC\bmu^\alpha(z)}{\cC\mu^\alpha(z)} \right)^*
={\mathbf 0}
\qquad\text{for }\mu^\alpha\ti{s}\text{-a.e.~}\lambda\in \partial \D.
\end{align}
By Poltoratski's Theorem, see Theorem  \ref{t-Polt}, applied entrywise to $\cC\bmu^\alpha/\cC\mu^\alpha$
we have
\begin{align}
\label{Cmu/C}
\lim_{z\nt\lambda}\left(\frac{\cC\bmu^\alpha(z)}{\cC\mu^\alpha(z)} \right)^* = W^\alpha(\lambda)
\qquad\text{for }\mu^\alpha\ti{s}\text{-a.e.~}\lambda\in\partial \D, 
\end{align}
where, recall $\dd \bmu^\alpha = W^\alpha \dd \mu^\alpha$. 

Take $\lambda\in\partial \D$ such that both \eqref{e-tozero} and \eqref{Cmu/C} are satisfied (it happens for $\mu^\alpha\ti{s}$-a.e.~$\lambda\in\partial \D$). 
Let $\be\in \ran W^{\alpha}\ti s (\lambda)$, so
$\be = W^\alpha(\lambda) \mathbff$  for some $\mathbff\in \C^n$. 
Then it follows from \eqref{Cmu/C} that 
\[
\lim_{z  \nt \la } \left(\be - \left(\frac{\cC\bmu^\alpha(z)}{\cC\mu^\alpha(z)} \right)^* \mathbff\right) =\bO , 
\]
and the uniform boundedness of $b(z)^*$ and \eqref{Cmu/C} imply that 
\[
\lim_{z \nt \la } (I_n -  \alpha b(z)^* ) \left(\be - \left(\frac{\cC\bmu^\alpha(z)}{\cC\mu^\alpha(z)} \right)^* \mathbff\right) =\bO . 
\]
Therefore by \eqref{e-tozero} we get that 
\[
\lim_{z \nt \la } (I_n -  \alpha b(z)^* )\be = \lim_{z \nt \la } (I_n -  \alpha b(z)^* )  \left(\frac{\cC\bmu^\alpha(z)}{\cC\mu^\alpha(z)} \right)^* \mathbff =\bO,
\]
so
\[
\lim_{z\nt\lambda} \alpha b(z)^*\be = \be. 
\]
Left multiplying the above identity by $\alpha^*$, we get the conclusion. 
\end{proof}

\begin{defn}\label{d-direcionalsupport}
For $\lambda \in \mathbb{T}$ we define the \emph{directional carrier} of $\bmu^\alpha\ti{s}$ by $${\mathbf S}_\alpha (\lambda) := \left\{{\be}\in \C^n:\lim_{z\nt\lambda}b(z)^*\be  = \alpha^*\be \right\},$$ wherever the non-tangential limit exists. Further, let $\bS(\lambda):= \bS\ci{I_n}(\lambda)$.
\end{defn}

\begin{lemma}
For any $\lambda\in\T$ there holds $\bS(\lambda) \perp (I_n-
\alpha^*)\bS_\alpha(\lambda)$. 
\end{lemma}
\begin{proof}
Let $\bff \in \bS(\lambda)$, $\be\in \bS_\alpha(\lambda)$.  It follows from the definition of $\bS_\alpha(\lambda)$ that for any ${\bf h}\in \spn\{\be, \bff\}$ the limit $\lim_{z\nt\lambda}b(z)^* \bh$ exists. Slightly abusing the notation let us call this limit $b(\lambda)^*\bh$; note that $b(\lambda)^*$ defined this way is a linear transformation acting from  $\spn\{\be, \bff\}$ to $\C^n$. 

It is easy to see that $b(\lambda)^*$ is a contraction. Moreover, since it acts isometrically on $\be$ and $\bff$,  it is an easy exercise to show that it acts isometrically on all of $\spn\{\be, \bff\}$. 

We know that $b(\lambda)^*\be=\alpha^*\be$, $b(\lambda)^*\bff=\bff$. 
Therefore,  
for $\mbf{e}\in \bS_\alpha(\la)$ and $\mbf{f}\in \bS(\la)$
we obtain
\[
(\mbf{e}, \mbf{f})_{\C ^n}
=
(b(\la)^*\mbf{e}, b(\la)^*\mbf{f})_{\C ^n}
=
(\alpha^*\mbf{e}, \mbf{f})_{\C ^n}.
\]
So $( (I_n-\alpha^*)\mbf{e}, \mbf{f})_{\C^n} =0 $ and we have $\mbf{f}\perp (I_n-\alpha^*) \mbf{e}.$
\end{proof}
\begin{remark}
We do not know about the existence of $\lim_{z \nt \la} b(z)^*$. But with respect to $\mu^\alpha\ti{s}$-a.e.~$\lambda\in\partial \D$ we have learned that $\lim_{z \nt \la} b(z)^*\be $ exists for every $\be \in \ran{W^\alpha(\lambda)}$. Slightly abusing notation and always being cautious about the meaning, we denote $b(\la)^* = \lim_{z \nt \la} b(z)^*$.
\end{remark}

Denote by $b(\lambda)$ any non-tangentional limit point  $\lim_{k\to\infty}b(z_k)$,  as $z_k \nt\lambda$; it exists because $\|b(z)\|\le 1$, but it does not have to be unique.

A generalization of a vector analog of the Aronszajn--Donoghue Theorem (on the mutual singularity of singular parts for rank one perturbations) follows without much effort from Proposition \ref{p-sing}.
To formulate this result, we recall the notion of vector mutual singularity (see \cite[Definition 6.1]{LT_JST}).

\begin{defn}\label{d-vector}
Matrix-valued measures $\bmu$ and $\bnu$ are said to be \emph{vector mutually singular}, $\bmu\perp\bnu$, if there exists a measurable function $\Pi$ with values in the orthogonal projections on $\mathbb{C}^n$ so that
\[
\Pi\bmu\Pi = \bO,
\qquad
(I-\Pi)\bnu(I-\Pi) = \bO,
\]
the matrix-valued zero measure.
\end{defn}
We note that for a measure $\dd \bmu = W\dd \mu$ and a Borel measurable matrix-valued function $\Pi$, the measure $\Pi\bmu\Pi$ is given by
\[
\Pi\bmu\Pi(E) = 
\int\ci{E}\Pi(z)^*[\dd\bmu(z)]\Pi(z)= 
\int\ci{E}\Pi(z)^*W(z)\Pi(z)\dd\mu(z)
\]
for Borel set $E\subset\mathbb{T}$.

It is not difficult to see that this definition can equivalently be formulated in terms of the densities of $\bmu$ and $\bnu$, if they are extended appropriately:
Two matrix-valued measures $\bmu$ and $\bnu$ are vector mutually singular if and only if there exist densities $W$ and $V$ with $\dd\bmu = W \dd\mu$ and  $\dd\bnu = V \dd\nu$ that satisfy $\ran{W}(z)\perp\ran{V}(z)$ for $(\mu+\nu)$-a.e.~$z\in \T$.

Proposition \ref{p-sing} has the following corollary.

\begin{cor}\label{c-orthog}
For unitary $\alpha$, we have 
$ 
\bmu\ti{s}\perp (I_n-\alpha^*)\bmu^\alpha\ti{s} (I_n -\alpha) .$
\end{cor}

\begin{remark}
Since the absolutely continuous part of any scalar measure is always mutually singular to any singular measure, we can drop the singular part on either of the matrix-valued measures in Corollary \ref{c-orthog}. So, for unitary $\alpha$, we have both
\begin{align*}
\bmu&\perp (I_n-\alpha^*)\bmu^\alpha\ti{s} (I_n -\alpha),\text{ and}
\\
\bmu\ti{s}&\perp (I_n-\alpha^*)\bmu^\alpha (I_n -\alpha).
\end{align*}
\end{remark}

\begin{proof}[Proof of Corollary \ref{c-orthog}]
Proposition \ref{p-sing} yields
\begin{align}
\label{ran W subset S}
\ran{W^\alpha(\la)} \subset {\bS}_\alpha(\lambda) \ \ \mu\ti s\text{-a.e.}
\quad\text{and}\quad
\ran{W(\la)} \subset \bS(\lambda)\ \ \mu^\alpha\ti s\text{-a.e.}, 
\end{align}
where $\bS(\lambda):=  \bS\ci{I_n}(\lambda)$.
In fact, we can always assume without loss of generality that the above inclusions \eqref{ran W subset S} holds $(\mu\ti s + \mu^\alpha\ti s)$-a.e.; we just need to pick appropriate representatives for densities $W$ and $W^\alpha$. 

To pick such representatives, let us notice that the measures $\mu$ and $\mu^\alpha$  are absolutely continuous with respect to the measure $\mu+ \mu^\alpha$, so 
\[
\dd \mu = u \dd  (\mu + \mu^\alpha), \qquad 
\dd \mu^\alpha = u^\alpha \dd  (\mu + \mu^\alpha) , 
\]
which implies
\[
\dd \mu\ti s = u \dd  (\mu\ti s + \mu^\alpha\ti s), \qquad 
\dd \mu^\alpha\ti s = u^\alpha \dd  (\mu\ti s + \mu^\alpha\ti s)  .
\]
Define $E=\{\xi\in \T: u(\xi)>0\}$, $E^\alpha=\{\xi\in \T: u^\alpha(\xi)>0\}$. Replacing the densities $W$ and $W^\alpha$ by $\1\ci{E}W$ and $\1\ci{E^\alpha}W^\alpha$ respectively, 
we do not change the measures $\dd \bmu = W\dd\mu$ and $\dd\bmu^\alpha= W^\alpha\dd\bmu^\alpha$. 

But for such choice of densities, any statement about $W$  that holds $\mu$-a.e.~or $\mu\ti s$-a.e.~also holds $(\mu+\mu^\alpha)$-a.e.~or $(\mu\ti s+\mu^\alpha\ti s)$-a.e.~respectively; and similarly for $W^\alpha$. So indeed, we can assume without loss of generality that inclusions \eqref{ran W subset S} holds $(\mu\ti s+\mu^\alpha\ti s)$-a.e. 

It follows from the definition of $\bS_\alpha(\lambda)$ that for $\lambda\in\T$ the limit $\lim_{z\nt\lambda}b(z)^*\be$ exists for all $\be\in \bS(\lambda)+\bS_\alpha(\lambda)$. This limit clearly defines a linear transformation from $\bS(\lambda)+\bS_\alpha(\lambda)$ to $\C^n$, which we, slightly abusing notation, will denote $b(\lambda)^*$. 

Clearly, $b(\lambda)^*$ is a contraction. Since by Proposition \ref{p-sing}
$\|b(\lambda)^* \be\|= \|\be\|$ for all $\be \in \bS(\lambda)$ and for all $\be \in\bS_\alpha(\lambda)$, we can conclude that $b(\lambda)^*$ acts isometrically on $\bS(\lambda)+\bS_\alpha(\lambda)$.

Therefore,  
for $\mbf{e}\in \bS_{\balpha} (\la)$ and $\mbf{f}\in \bS(\la)$
we obtain
\[
(\mbf{e}, \mbf{f})_{\C ^n}
=
(b(\la)^*\mbf{e}, b(\la)^*\mbf{f})_{\C ^n}
=
(\alpha^*\mbf{e}, \mbf{f})_{\C ^n}.
\]
So $( (I_n-\alpha^*)\mbf{e}, \mbf{f})_{\C^n} =0 $ and we have $\mbf{f}\perp (I_n-\alpha^*) \mbf{e}.$

We have shown that inclusions \eqref{ran W subset S} hold for $(\mu^\alpha+\mu)\ti{s}$-a.e.~$\lambda\in \mathbb{T}$, so
\[
\ran{W (\la)}\perp (I_n-\alpha^*)\ran{W^\alpha(\la)}
\]
for $(\mu^\alpha+\mu)\ti{s}$-a.e.~$\lambda\in \mathbb{T}$, and that is equivalent to the statement.
\end{proof}

%%%%%%%%%%%%%%%%%%%%%%%%%%%%%%%%%%%%%%%%%%%%%%%%%%%%%%%%%%%%%%%%%%%%%%%%%%%%%%%%%%%%%%%%%%
\section{Strong mutual singularity}\label{s-traces}
%%%%%%%%%%%%%%%%%%%%%%%%%%%%%%%%%%%%%%%%%%%%%%%%%%%%%%%%%%%%%%%%%%%%%%%%%%%%%%%%%%%%%
The vector mutual singularity from Corollary \ref{c-orthog} is used to show a strong mutual singularity for the traces.
\begin{thm}
\label{t:countable_A-D}
Let  $\alpha:\R\to \cU(n)$  %$t\mapsto\alpha(t)$ 
be a $C^1$ function such that for all $t\in\R$ its ``logarithmic derivative'' $i\alpha'(t)\alpha(t)^{-1}$ is sign definite. Then, given any singular Radon measure $\nu$ on $\R$, the scalar measures $\mu^{\alpha(t)} := \tr \bmu^{\alpha(t)}$ are mutually singular with $\nu$ for all $t\in\R$ except probably countably many. 
\end{thm}

\begin{remark*}
Note that if $\alpha(t)\in\cU(n)$ for all $t$, the ``logarithmic derivative'' $i\alpha'(t)\alpha(t)^{-1}$ is always Hermitian. It follows, for example from the description of the tangent space to $\cU(n)$; an elementary proof is also easy. 

Note also that the matrices $\alpha(t)^{-1}\alpha'(t)$ and $\alpha'(t) \alpha(t)^{-1}$ are unitarily equivalent, so in the above Theorem \ref{t:countable_A-D}  we can use the condition that the matrix $i\alpha(t)^{-1}\alpha'(t)$ is sign definite. 
\end{remark*}

\begin{lemma}
\label{l:A-ort-distance}
Let $A=A^*$ be a sign definite matrix. Then for a sufficiently small $\delta>0$ for any matrix $\wit A$ (not necessarily Hermitian) such that $\|A-\wit A\|<\delta$ the condition $(\wit A\bx, \by)_{\C^n}=0$ implies that 
\[
\|\bx-\by\|^2 \ge c \cdot \left( \|\bx\|^2+\|\by\|^2 \right), 
\]
where $c=c(A, \delta)$.
\end{lemma}
\begin{proof}
Replacing the norm in $\C^n$ by an equivalent one we can assume without loss of generality that $A=I$. The condition $\|I-\wit A\|<\delta$ together with the assumption $(\wit A\bx, \by)_{\C^n}=0$ imply that $|(\bx, \by)_{\C^n} | \le \delta \|\bx\|\cdot\|\by\|$. Then for $\delta<1$
\begin{align*}
\|\bx-\by\|^2 & = \|\bx\|^2 + \|\by\|^2 - 2\re (\bx, \by)_{\C^n} \\
&  \ge \|\bx\|^2 + \|\by\|^2 -2\delta \|\bx\|\cdot\|\by\|  \\
& \ge c(\delta) \left( \|\bx\|^2 + \|\by\|^2 \right).
\end{align*}
The lemma follows.
\end{proof}

\begin{proof}[Proof of Theorem \ref{t:countable_A-D}]
Fix $t_0\in\R$. Differentiability of $\alpha$ implies that 
\begin{align*}
\alpha(t) - \alpha(t_0) = \left( \alpha'(t_0) +o(1) \right) \cdot (t-t_0) \qquad \text{as }t\to t_0.
\end{align*}
Therefore, for any $\e>0$ there exists an open neighborhood $U\ni t_0$ such that for any $t, t'\in U$ we have 
\begin{align*}
\alpha(t') - \alpha(t)   = \left( \alpha'(t_0) +r(t',t) \right) \cdot (t'-t), \qquad \| r(t',t)\|<\e. 
\end{align*}
Continuity of $\alpha$ implies that for any $\delta>0$ we can find a neighborhood $\wit U\ni t_0$ such that for all $t, t'\in\wit U$
\begin{align}
\label{d_alpha} 
\alpha(t') \alpha(t)^{-1} - I = \left( \alpha'(t_0)\alpha(t_0)^{-1} +\wit r(t',t) \right) \cdot (t'-t), \qquad \| \wit r(t',t)\|<\delta.
\end{align}
One of the operators $\pm i \alpha'(t_0)\alpha(t_0)^{-1}$ is positive definite. Pick sufficiently small $\delta$ in \eqref{d_alpha} such that Lemma \ref{l:A-ort-distance} will apply to $A=\pm i\alpha'(t_0)\alpha(t_0)^{-1}$  with some $c>0$. 

Let $t\in \wit U$ be such that the scalar measure $\mu^{\alpha(t)}$ is not mutually singular with $\nu$. Let $\wt \mu^{\alpha(t)}$ and $\wt \mu^{\alpha(t')}$ be the absolutely continuous with respect $\nu$ parts of $\mu^{\alpha(t)}$,   $\dd \wt\bmu^{\alpha(t)}\ti{s} = \wt W^{\alpha(t)}\dd\nu $. Take a function $f_t\in L^2(\nu)$, $\|f_t\|\ci{ L^2(\nu)}=1$, such that 
\begin{align}
\label{range iclusion 01}
f_t(\xi) \in \ran \wt W^{\alpha(t)} (\xi)  \qquad \text{for } \nu\text{-a.e.~}\xi\in \T.
\end{align}

Now, let $t,t'\in\wit U$ be such that both scalar measures $\mu^{\alpha(t)}$, $\mu^{\alpha(t')}$ are not mutually singular with $\nu$.

By Corollary \ref{c-orthog} we have $\bmu^{\alpha(t)}\ti{s}\perp (I_n-\alpha^*)\bmu^{\alpha(t')}\ti{s} (I_n -\alpha)$, where we used $\alpha= \alpha(t')\alpha(t)^{-1}$, and so
\begin{align*}
\wt\bmu^{\alpha(t)}\ti{s}\perp (I_n-\alpha^*)\wt\bmu^{\alpha(t')}\ti{s} (I_n -\alpha) .
\end{align*}
This implies implies 
\begin{align*}
\ran \wt W^{\alpha(t)} (\xi)\perp \ran ( (I_n-\alpha^*) \wt W^{\alpha(t)} (\xi) )\qquad \text{for } \nu\text{-a.e.~}\xi\in \T. 
\end{align*}
So for the functions functions $f_t, f_{t'}\in L^2(\nu)$, $\|f_t\|\ci{L^2(\nu)}=\|f_{t'}\|\ci{L^2(\nu)}=1$ defined above in \eqref{range iclusion 01}  we have, 
\begin{align*}
f_t(\xi) \perp (I_n-\alpha^*) f_{t'}(\xi) \qquad \text{for } \nu\text{-a.e.~}\xi\in \T.
\end{align*}
Lemma \ref{l:A-ort-distance} implies that 
\[
\|f_t(\xi) - f_{t'}(\xi) \|\ci{\C^n}^2 \ge c\cdot (\|f_t(\xi)\|\ci{\C^n}^2 +\| f_{t'}(\xi) \|\ci{\C^n}^2 ) \qquad \text{for } \nu\text{-a.e.~}\xi\in \T, 
\]
and integrating with respect to $\dd \nu(\xi)$ we get that 
\[
\|f_t- f_{t'}\|\ci{L^2(\nu)}^2 \ge c. 
\]
So, by the separability of $L^2(\nu)$, the measures $\mu^{\alpha(t)}$ can be not mutually singular with $\nu$ only for countably many $t\in \wit U$. 

Using standard compactness reasoning we get that any compact $K\subset \R$ can have at most countably many such $t$'s, and covering $\R$ by countably many compacts we get the conclusion of the theorem. 
\end{proof}

\section{Carath\'{e}odory condition}\label{s-CAD}
This section will lay the ground work for the investigation of the Carath\'{e}odory angular derivative that will be done later in Sections \ref{s-CAD2} and \ref{s-CaraPoint}.

Recall that for a function $b\in\scr S(n)$ the de Branges--Rovnyak space $\scr H(b)$ is defined as follows. Let $T_b: H^2(\C^n)\to H^2(\C^n)$ be the (analytic) Toeplitz operator, 
\[
T_b f := bf, \qquad f\in H^2 (\C^n) 
\]
and $T_b^*$ be its adjoint. 

Then the de Branges--Rovnyak space $\scr H(b)$ is the range of the operator $R_b:= (I_n - T_b T_b^*)^{1/2}$ endowed with the 
\emph{range norm}, 
\[
\| f\|\ci{\scr H(b)} = \inf\left\{ \|h\|\ci{H^2(\C^n)} : h\in H^2(\C^n) \text{ such that } R_b h = f\right\} . 
\]
Clearly, $\scr H(b) \subset H^2(\C^n)$ and $\| f \|\ci{\scr H(b)}  \ge \| f \|\ci{H^2(\C^n)}.$

The matrix-valued function $k^b(z,w)=: k^b_w(z)$ defined in \eqref{d-kernel}, 
is the \emph{matrix reproducing kernel} for $\scr H(b)$, meaning that for any $w\in \D$ and $\be\in\C^n$
\begin{align}\label{e-reproducing}
( f(w), \be)\ci{\C^n} = \left\langle  f, k^b_w \be \right\rangle_{\scr H(b)}.
\end{align}
Finally, let us mention that the linear combinations of functions $k^b_w\be$, $w\in\D$, $\be\in \C^n$ are dense in $\scr H(b)$. Indeed, if $f\in\scr H(b)$ is orthogonal to all $k^b_w\be$, the reproducing kernel property \eqref{e-reproducing} implies that $f\equiv  0$.

\begin{defn}\label{d-CCond}
We say  that a function $b\in\scr{S}(n)$ satisfies the Carath\'{e}odory condition in the codirection $\bx\in\C^n$ at a point $\lambda\in \T$ if 
\begin{align}
\label{e:Caratheodory cond}
\liminf_{z\nt \lambda} \frac{\|\bx\|^2 - \|b(z)^*\bx\|^2} {1-|z|^2} <\infty. 
\end{align}
\end{defn}

\begin{prop}\label{p-equivalences}
Given $b \in \scr{S} (n)$ and a non-zero vector $\mbf{x} \in \C ^n$ and $\la \in \partial \D$, the following are equivalent:
\bn
\item The function $b$ satisfies Carath\'{e}odory condition in the codirection $\bx$ at the point $\lambda$.

\item There exists $\wit\bx \in\C^n$ 
such that the function 
\begin{align*}
\frac{\bx - b(z) \wit\bx}{1-z\overline{\lambda}}
\end{align*}
belongs to $H^2(\C^n)$.

\item For any $f\in \scr H(b)$ the limit 
\[
\lim _{z\nt \la} \ipcn{\mbf{x}}{f(z )} =:  
\ell _{\la , \mbf{x}} (f)
\]
exists and the    
linear functional $f\mapsto \ell _{\la , \mbf{x}} (f) 
$ is bounded on $\scr{H} (b)$.

\item A stronger version of the Carath\'{e}odory condition holds, i.e. 
\begin{align*}
\limsup_{z\nt \lambda} \frac{\|\bx\|^2 - \|b(z)^*\bx\|^2} {1-|z|^2} <\infty. 
\end{align*}
\en

Moreover, if the above conditions are satisfied, then 
\begin{align*}
\lim_{z\nt \lambda} b(z)^*\bx =: b^*(\lambda, \bx)
\end{align*}
exists and equals to $\wit\bx$ from Statement \cond2, the function 
\begin{align}\label{e-brpk}
k^b_{\lambda, \bx} (z) := \frac{\bx - b(z)b^*(\lambda, \bx)}{1-z\overline\lambda}
\in \scr H(b)
\end{align}
and 
the linear functional $\ell _{\la , \mbf{x}}$ is given by $ \ell _{\la , \mbf{x}} (f) = \langle f, k^b_{\lambda, \bx} \rangle\ci{\scr H(b)}$.
\end{prop}

\begin{defn}
\label{d: bdry repr kern}
We call the function in \eqref{e-brpk} the boundary reproducing kernel of $\scr H(b)$ in the codirection $\bx$ at the point $\lambda\in\T$. 
\end{defn}

\begin{proof}
Let Statement \cond1 be satisfied. Notice, that  this condition just means 
\[
\liminf _{z \nt \la} \| k_z ^b \mbf{x} \|_{\scr H(b)} ^2 =: C < \infty. 
\]
Hence there is a sequence $z_k \nt \la$ so that $\| k _{z_k} ^b \mbf{x} \|_{\scr H(b)} ^2 \rightarrow C$, and we can also assume, without loss of generality (by passing to a subsequence, if necessary) that $k_{z_k}^b \mbf{x}$ converges weakly to some $h \in \scr{H} (b)$ by weak compactness. 
Similarly, we can also assume that $b(z_k)\bx$ converges to some vector $\wit\bx \in \C^n$. 
Notice that the Carath\'{e}odory condition implies that $\|\wit\bx\|=\|\bx \|$.

Then for any $\by\in\C^n$
\ba \ipcn{\mbf{y}}{h(z)} & = & \ip{k_{z} ^b \mbf{y}}{h}\ci{\scr H(b)} \nn \\
& = & \lim _{k \rightarrow \infty} \ip{k_{z} ^b \mbf{y}}{ k_{z_k} ^b \mbf{x}}\ci{\scr H(b)} \nn \\
& = & \lim_{k \rightarrow \infty} \ipcn{\mbf{y}}{\frac{I_n - b(z)b(z_k ) ^*}{1-z\ov{z}_k} \mbf{x}} \nn \\
& = & \ipcn{\mbf{y}}{\frac{\bx - b(z) \wit\bx}{1-z\bar\la} },  \nn 
\ea 
so 
\begin{align}
\label{e:h(z)}
h(z)= \frac{\bx -\wit\bx}{1-z\bar\la}.
\end{align}

Since $\scr{H}(b)\subset H^2(\C^n)$, 
we get that $h\in H^2(E)$, and the above formula for $h$ implies Statement \cond2. 

Assuming now that Statement \cond2 holds, let us prove \cond3. First let us show that 
\begin{align}
\label{e:lim b(z)* y}
\lim_{z\nt \lambda} b(z)^*\bx =\wit\bx. 
\end{align}
By the assumption \cond2 we know that $h$ defined by \eqref{e:h(z)} belongs to $H^2(\C^n)$. Every $H^2$ function is $O((1-|z|)^{-1/2})$. Since as $z\nt\lambda$ the quantities $1-|z|^2$ and $|1-\overline{\lambda}z|$ are comparable in the sense of two sided estimates, we conclude that for $z\nt\lambda$
\begin{align*}
\by - b(z)\wit\bx & = O((1-|z|)^{1/2}), \\
\intertext{or, equivalently}
b(z)^* \bx - \wit\bx & = O((1-|z|)^{1/2}).
\end{align*}
But the right hand side tends to $0$ as $z\nt \lambda$, so \eqref{e:lim b(z)* y} is proved. 

Statement \cond3 will follow immediately from the weak convergence $k_z^b \bx \to h$ in $\scr H(b)$ as $z\nt \lambda$.

To prove the weak convergence, notice first that identity \eqref{e:lim b(z)* y} implies that for $h(z)$ given by \eqref{e:h(z)} $k_z^b(w)\bx \to h(w)$,  as $z\nt \lambda$ for all $w\in\D$. So, for all $w\in\D $ and all $\by\in \C^n$ we have $ \langle k_{z}^b \bx, k_w^b \by \rangle\ci{\scr{H}(b)} \to \langle h, k_w^b \by \rangle\ci{\scr{H}(b)} $ as $z\nt\lambda$. 

Therefore, to prove the weak convergence, it is sufficient to show that $$\|k_z^b\bx\|\ci{\scr H(b)} \le C<\infty$$ for $z$ in the non-tangential approach region, because 
we already have the convergence on a dense set (linear combinations of $k_w^b \by$).

Since $\|b(z)\|\le 1$ and $\|\bx\|=\|\wit\bx\|$, we have
\ba 
0 & \le & 2 \re{\ipcn{\mbf{x}}{(\bx - b(z) \wit\bx ) }}\nn \\
& = & \ipcn{\mbf{x}}{(I_n - b(z) b(z)^* ) \mbf{x}} + \| \wit\bx -  b(z) ^* \mbf{x} \| ^2 \ci{\C ^n} \nn .
\ea 
It follows that: 
\ba \| k_z ^b \mbf{x} \| ^2 & = & \frac{1}{1-|z| ^2} \ipcn{\mbf{x}}{(I_n-b(z) b(z) ^*) \mbf{x}} \nn \\ 
& \leq & \frac{2}{1-|z| ^2} \left| \ipcn{\bx}{\by -b(z) \wit \bx} \right| \nn \\
& = & 2 \frac{|1-z\bar\la|}{1-|z|^2} \left| \ip{k_z ^b \mbf{x}}{h}\ci{\scr H(b)} \right| \nn \\
& \leq & 2 \frac{|1-z\bar\la|}{1-|z|^2} \| k_z ^b \mbf{x} \| \| h \|. \nn \ea 
And so $\| k _z ^b \mbf{x} \| $ is uniformly bounded in any non-tangential approach region for $\la$. 

Finally, let us prove \cond3$\implies$\cond4. 
Suppose that $\ell _{\la , \mbf{x}}$ is a bounded linear functional and consider a Stolz domain $\Ga _t (\la)$. Then for any $f \in \scr{H} (b)$, 
$$ \sup _{z \in \Ga _t (\la )}  \{ | \ip{k_z ^b \mbf{x}}{f}\ci{\scr H(b)} | \} < \infty. $$ 
By the Principle of Uniform Boundedness it then follows that 
$$ \sup _{z \in \Ga _t (\la ) } \| k _z ^b \mbf{x} \| < \infty
, $$ 
which is exactly Statement \cond4.
\end{proof}

\begin{prop}
\label{p:CAD subspace} For $b\in\scr S(n)$ and $\lambda\in\partial\D$, the set of all \emph{Carath\'{e}odory codirections} at $\lambda\in\T$, i.e.~the collection of all $\bx \in\C^n$ such that $b$ satisfies the Carath\'{e}odory condition at $\lambda$ is a subspace of $\C^n$.
\end{prop}

\begin{proof}
The  set of all Carath\'{e}odory codirections is invariant under multiplication by scalars, so it is enough to show that if $\bx$ and $\by$ are Carath\'{e}odory codirections, then so is $\bx + \by$. 
Since $\|b(z)\|\le 1$, we conclude
\begin{align*}
0& \le \|\bx\pm \by\|^2 - \|b(z)^*(\bx\pm \by)\|^2\\
& = \|\bx\|^2 +\|\by\|^2 -  \|b(z)^*\bx\|^2 -  \|b(z)^*\by\|^2 \pm 
2\re \left(  (\bx, \by) - ( b(z)^*\bx, b(z)^*\by )    \right). 
\end{align*}
Since this expression is non-negative for any choice of $\pm$ sign, we see that 
\begin{align*}
2 \left| \re \left(  (\bx, \by)\ci{\C^n} - ( b(z)^*\bx, b(z)^*\by )\ci{\C^n}    \right)          \right| 
\le \|\bx\|^2 +\|\by\|^2 -  \|b(z)^*\bx\|^2 -  \|b(z)^*\by\|^2 , 
\end{align*}
and therefore 
\begin{align*}
\|\bx\pm \by\|^2 - \|b(z)^*(\bx\pm \by)\|^2 \le 2 \left(  \|\bx\|^2 - \|b(z)^*\bx\|^2 +  \|\by\|^2 - \|b(z)^*\by\|^2  \right) .
\end{align*}
Since both $\bx$ and $\by$ are Carath\'{e}odory codirections, this inequality and part $\mr{(iv)}$ of Proposition \ref{p-equivalences} imply that $\bx\pm \by$ are also Carath\'{e}odory codirections. 
\end{proof}

%%%%%%%%%%%%%%%%%%%%%%%%%%%%%%%%%%%%%%%%%%%%%%%%%%
\section{Carath\'{e}odory angular derivative (CAD)}\label{s-CAD2}
%%%%%%%%%%%%%%%%%%%%%%%%%%%%%%%%%%%%%%%%%%%%%%%%%%%
\begin{defn}\label{d-CAD}
Let $E$ be a subspace of $\C^n$ and denote by $P\ci{E}:\C^n\to E$ the corresponding orthogonal projection. A function $b \in \scr{S} (n)$ is said to have a 
Carath\'{e}odory angular derivative%
\footnote{Or simply \emph{angular derivative}. }
(CAD) at $\la \in \T$ \emph{on the subspace}  $E$, if for every $\be\in E$ 
\begin{align}
&\lim_{z\nt \lambda}b(z) \be = \be, \label{e-limite}\intertext{and}
&\CAD_E b (\lambda):=    
\lim_{z\nt \lambda}P\ci{E} b'(z)P\ci{E}\quad\text{ exists.}\label{e-limit derivative}
\end{align}
\end{defn}

\begin{prop}
A function $b\in \mathscr S (n)$ has a Carath\'{e}odory angular derivative on a  subspace  $E$ if and only if 
\begin{align}
\lim_{z\nt \lambda}P\ci{E}\frac{b(z)-I}{z-\lambda}P\ci{E}\quad\text{ exists.}\label{e-limitoperator}
\end{align}
Moreover, in this case the limits \eqref{e-limitoperator} and \eqref{e-limit derivative} coincide. 
\end{prop}

\begin{proof}
The proof goes exactly as in the scalar-valued case, and is only presented for the convenience of the reader. 

Suppose $f$ has a Carath\'{e}odory angular derivative at $\lambda$ on a subspace $E$. For $w\in\D$
\[
P\ci E (b(z) - b(w)) P\ci E = \int_w^z b'(\xi)\dd\xi . 
\]
Now take $z$ in the non-tangential approach region (Stolz domain $\Gamma_t(\lambda)$), and let $w\to \lambda$ along the line connecting $z$ and $\lambda\in\T$. Then  using \eqref{e-limite} and, say, the dominated convergence theorem we get that 
\begin{align*}
P\ci E (b(z) - b(\lambda)) P\ci E = \int_\lambda^z  P\ci E b'(\xi) P\ci E \dd\xi;
\end{align*}
for simplicity we can assume here that the integral is taken over the interval connecting $z$ and $\lambda$. Then trivially using \eqref{e-limit derivative} we get
\begin{align*}
P\ci E (b(z) - b(\lambda)) P\ci E - (z-\lambda)\CAD_E b(\lambda)  & =  \int_\lambda^z 
\left(  P\ci E b'(\xi) P\ci E -\CAD_E b(\lambda) \right) \dd\xi \\
&= o(1) (z-\lambda)  . 
\end{align*}
Dividing by $z-\lambda$ and taking the limit as $z\to\lambda$ we see that the limit \eqref{e-limitoperator} exists and coincides with $\CAD_E b(\lambda)$. 

Vice versa, assume now that the limit \eqref{e-limitoperator} exists. Let us call this limit  $A$. Take $\be\in E\subset \C^n$ and denote $f(z) := P\ci E b(z)\be$.  Fix a Stolz region $\Gamma_t(\lambda)$, and a bigger region $\Gamma_{t'}(\lambda)$, $t'>t$. For each $z\in\Gamma_t(\lambda)$ let $T_z$ be the circle centered at $z$ of radius $\delta\cdot (1-|z|)$, where $\delta>0$ is sufficiently small so $T_z\subset \Gamma_{t'}(\lambda)$ for all $z\in \Gamma_t(\lambda)$. 

Then for $z\in\Gamma_t(\lambda)$
\begin{align*}
f'(z)   & =\frac{1}{2\pi i } \int_{T_z} \frac{f(\xi)}{(\xi-z)^2} \dd\xi   \\
& = \frac{1}{2\pi i } \int_{T_z} \frac{f(\xi) - \be}{(\xi-z)^2} \dd\xi   \\
& = \frac{1}{2\pi i } \int_{T_z} \left( \frac{f(\xi) - \be}{\xi-z} - A\be \right) \frac{1}{\xi-z} \dd\xi 
+ \frac{1}{2\pi i } \int_{T_z} \frac{A\be}{\xi-z} \dd\xi . 
\end{align*}
Now, let $z\nt \lambda$ (with $z\in \Gamma_t(\lambda)$). By the Cauchy theorem the second integral in the last line is always $A\be$. And since the limit in \eqref{e-limitoperator} is $A$, we see that as  the first integral tends to $0$. 
\end{proof}

In the scalar case ($n=1$) Definition \ref{d-CAD} covers only special case when the non-tangential boundary value  $b(\lambda) = \lim_{z\nt\lambda} b(z) =1$. The definition below covers the general case. 

\begin{defn}
Let $E$ be a subspace of $\C^n$ and $\alpha\in\cU(n) $. 
We say that a function $b\in\scr S(n)$ has a Carath\'{e}odory angular  derivative in \emph{codirection} $(E, \balpha)$ at $\lambda\in\partial\D$ if $b\balpha^*$ has a Carath\'{e}odory angular  derivative on a subspace $E$.
\end{defn}

\begin{thm}
\label{t:CAD}
If $\bmu^\alpha\{\la\}\neq 0$ and $\ran{\bmu^\alpha\{\la\}}=E$ for some $\lambda \in \T$, then $b\alpha^*$ has a 
Carath\'{e}odory angular derivative  $\CAD\ci E (b\alpha^*)(\lambda)$ at $\lambda$ on the subspace $E$, and its Moore--Penrose inverse equals $\lambda\bmu^\alpha\{\la\}$.
\end{thm}

\begin{remark}
\label{r:b(0)=0}
The functions $b(z)$ and $\wt b(z):=\bar\lambda z b(z)$ have the same boundary behavior at $\lambda\in \T$, meaning that the Aleksandrov--Clark measures for $b$ and $\wt b$ have the same point mass at $\lambda$, and the limits in the definition of the %restricted 
CAD for both functions also coincide. 

Therefore, replacing $b$ by $\wt b$ we can assume without loss of generality that  $b(0)=0$.
\end{remark}

\subsection{Some known facts about the Clark operator}
\label{s:Clark}
To prove the theorem, let us first recall some results about the Clark model  obtained in \cite{LTFinite}. Note first that for the case $b(0)=0$ the matrix-valued  measure $B^*B\dd\mu$  considered in \cite{LTFinite} coincides with the Aleksandrov--Clark measures treated in this paper (in the case $b(0)\ne 0$ they differ by normalization). 

In \cite{LTFinite} the Clark operator $\Phi$ from the model space $\cK_b$ and the weighted $L^2$ space $L^2(\bmu)$, where matrix-valued measure $\bmu$ is the Aleksandrov--Clark measure for $b$,  was constructed.  

Let us recall that for a matrix valued measure $\bmu =W \mu$, the norm in the weighted space $L^2(\bmu)$ is defined as
\begin{align}
\label{e:L^2(bmu)}
\| f\|\ci{L^2(\bmu)}^2= \int \left( \dd\bmu (\xi) f(\xi), f(\xi)  \right)\ci{\C^n} =    
\int \left( W (\xi) f(\xi), f(\xi)  \right)\ci{\C^n} \dd\mu(\xi)
\end{align}
and the space $L^2(\bmu)$ consists of the Borel measurable vector-valued functions $f$ with $\| f\|\ci{L^2(\bmu)}<\infty$.
We should mention here, that while for general operator-valued measures, the definition of the weighted space $L^2(\bmu)$ is quite involved, a matrix-valued measure $\bmu$ can always be represented as $\dd\bmu = W\dd\mu$ with a scalar measure $\mu$ (for example with $\mu=\tr\bmu$), and the integration is reduced to the standard integration with respect to the scalar measure $\mu$, see \eqref{e:L^2(bmu)}.

Let us also recall, that the  Sz.-Nagy--Foia\c s model space $\cK_b$ is defined as the set of vector-valued functions
\begin{align*}
\cK_b:=
\left( 
\begin{array}{c}
H^2(\C^n)  \\
\clos \Delta L^2(\C^n)
\end{array}
\right)
\ominus
\left( 
\begin{array}{c}
b  \\
\Delta
\end{array}
\right) H^2(\C^n),
\end{align*}
where $\Delta (\xi):= (I_n - b(\xi)^*b(\xi))^{1/2}$, $\xi\in \T$; here by $b(\xi)$ we mean the non-tangential boundary  values of $b(z)$ as $z\nt \xi$, $z\in\D$. 

In \cite[Theorem 8.7]{LTFinite} the following description of the adjoint Clark operators $\Phi^*$ was obtained (we present only formula for the case $b(0)=0$ here)
\begin{align}
\label{Phi*}
\Phi^* f = \left(   \begin{array}{c}
0  \\
\Delta (\cC\bmu)_+ 
\end{array}   \right) f
+
\left(   \begin{array}{c}
(\cC \bmu)_+^{-1}  \\
\Delta 
\end{array}   \right) (\cC[\bmu f])_+ \,.
\end{align}
Here, as was defined in \eqref{d-Cauchy}, $\cC\bmu$ is the Cauchy transform in the unit disc $\D$ of the measure $\bmu$. In analogy, $\cC[\bmu f]$ is the Cauchy transform of the vector-valued measure $\bmu f$. The subscript ``$+$''  means that we are considering non-tangential boundary values from $\D$ to the unit circle $\T$.

Theorem 8.7 stated in \cite{LTFinite} using a different notation,  so for the convenience of the reader we provide the translation. 

The function $\theta=\theta\ci\Gamma$ corresponds to the function $b$ in the present paper. The matrix $\Gamma$ is defined as $\Gamma=-\theta(0)$, so in our case $\Gamma=0$. Symbols $D\ci\Gamma$ and $D\ci{\Gamma^*}$ denote the \emph{defect operators}, 
\[
D\ci\Gamma = (I-\Gamma^*\Gamma)^{1/2}, \qquad D\ci{\Gamma^*} = (I-\Gamma\Gamma^*)^{1/2}, 
\]
so in our case $D\ci\Gamma = D\ci{\Gamma^*} = I_n $. 

The matrix-valued measure $B^*B\mu$, where $\mu$ is a scalar measure and $B$ is a matrix-valued function corresponds to our Aleksandrov--Clark measure $\bmu$. To see that one can compare  \cite[Equation (4.6)]{LTFinite} ($\theta_0$ there means that $\theta(0)=0$)   with \eqref{e-Herglotz}, noticing that if $b(0)=0$ then $H(0)=I_n$, so $\im H(0) =0$.  

Finally, $\Delta\ci\Gamma$ in \cite{LTFinite} corresponds to our $\Delta$. 

We do not need all properties of $\Phi$ (or of $\Phi^*$) that were presented in \cite{LTFinite}; we only need to know that formula \eqref{Phi*} defines a unitary operator $\Phi^*: L^2(\bmu) \to \cK_b$.

\subsection{Proof of Theorem \ref{t:CAD} }
As discussed in Remark \ref{r:b(0)=0}, without loss of generality, we can assume  $\alpha=I_n.$ 

The operator $\Phi^* : L^2(\bmu) \to \cK_b$  is a unitary operator. However, in its definition \eqref{Phi*} the model space $\cK_b$ is not involved. So, if we increase the target space from $\cK_b $ to $L^2(\C^n) \oplus L^2(\C^n)$, we can treat the operator $\Phi^*$, defined by the formula \eqref{Phi*} as an operator from  $L^2(\bmu)$ to $ L^2(\C^n) \oplus L^2(\C^n)$.  To avoid confusion, let us denote this operator with extended target space as $\wt\Phi^*$, and let $\wt\Phi: L^2(\C^n) \oplus L^2(\C^n) \to L^2(\bmu)$ be its adjoint. 

One can easily see that $\wt\Phi^*: L^2(\bmu) \to  L^2(\C^n) \oplus L^2(\C^n)$ is an isometry, and that $\ran \wt\Phi^* = \cK_b$. 

Define $\wt\Phi_m:= \wt\Phi - M_\xi^m\wt\Phi M_{\bar z}^m$. Here we use $\xi, z\in \T$ for independent variables in $L^2(\bmu)$ and $ L^2(\C^n) \oplus L^2(\C^n)$ respectively, so
\[
[M_\xi f](\xi) := \xi f(\xi), \qquad f\in L^2(\bmu), 
\]
and 
\[
( M_{\bar z} g)(z) = \bar z g(z), \qquad g\in L^2(\C^n) \oplus L^2(\C^n). 
\]
Take $f = \1\ci{\{\lambda\}}\mbf{e}$, $\be\in\ran{\bmu\{\la\}}$ and $g\in L^2(\C^n) \oplus L^2(\C^n)$. Since trivially $M_\xi^* f = M_{\bar\xi} f = \bar\lambda f$, we see that 
\[
(M_{\xi}^m \wt\Phi M_{\bar z}^m g, f)\ci{L^2(\bmu)} =  ( M_{\bar z}^m g, \wt\Phi^* M_{\bar\xi}^m f)\ci{L^2} = 
\lambda^m ( M_{\bar z}^m g, \wt\Phi^*  f)\ci{L^2} \longrightarrow 0,
\]
as $m\to\infty$  (because $M_{\bar z}^m g \to 0 $  weakly in $L^2(\C^n) \oplus L^2(\C^n)$).

We then can  conclude that (for $f = \1\ci{\{\lambda\}}\mbf{e}$, $\be\in\ran{\bmu\{\la\}}$ and $g\in L^2(\C^n) \oplus L^2(\C^n)$) we have the limiting behavior
\[
(\wt\Phi_m g, f)\ci{L^2(\bmu)}
\longrightarrow 
(\wt\Phi g, f)\ci{L^2(\bmu)}
\qquad\text{as }m\to \infty.
\]
If we now take $g=\wt\Phi^*( \1\ci{\{\lambda\}}  \be')$, $\be'\in \ran \bmu\{\la\}$, then using the fact that $\wt\Phi^*$ is an isometry, we get that
\[
(\wt\Phi g, f)\ci{L^2(\bmu)}
=
\left(\1\ci{\{\lambda\}} \be', \1\ci{\{\lambda\}}\be\right)\ci{L^2(\bmu)}
=
(\bmu\{\la\} \be', \be)\ci{\C^n}.
\]

Using formula \eqref{Phi*} we get that
\begin{align*}
\wt\Phi_m^* f(z) = \Phi^* f (z) - M_z^m \Phi^*( M_{\bar\xi}^m f)  = C_1(z) \int_\T \frac{1-(z\bar\xi)^m}{ 1-z\bar\xi} [\bmu(\dd\xi)] f(\xi);
\end{align*}
where 
\[
C_1 = \left(   \begin{array}{c}
(\cC \bmu)_+^{-1}  \\
\Delta 
\end{array}   \right)\,.
\]
Therefore, for $g\in L^2(\C^n) \oplus L^2(\C^n)$
\begin{align*}
\wt \Phi_m g(\xi) = \int_\T \frac{1-(\bar z\xi)^m}{ 1-\bar z\xi} C_1(z)^* g(z) \dd m(z).
\end{align*}
It was shown in \cite[Lemma 8.6]{LTFinite} that $(\cC\bmu)(z) = (I_n - b(z))^{-1}$ for all $z\in\D$ and a.e.~on $\T$; this fact can also be easily obtained from \eqref{Herglotz fn}. Therefore, if $g=g_1\oplus g_2 \in \cK_b$ (note that $g_1\in H^2(\C^n)$) 
\[
C_1^* g = g_1 + (b^* g_1 + \Delta g_2)\qquad \text{a.e.~on }\T. 
\]
By the definition of $\cK_b$ we have $h:= (b^* g_1 + \Delta g_2) \in H^2_-(\C^n)$, so
\[
\int_\T \frac{1-(\bar z\xi)^m}{ 1-\bar z\xi} h(z) \dd m(z) =0
\qquad \text{for all }m\in \N_0. 
\]
For a function $f\in H^2$ (possibly vector-valued) denote by $P_m f$ the partial sum of its Taylor (Fourier) series, i.e.~$P_m f(z) = \sum_{k=0}^m \hat f (k) z^k$.
With this notation we have
\begin{align}
\label{Phi_m}
\wt \Phi_m g(\xi) = \int_\T \frac{1-(\bar z\xi)^m}{ 1-\bar z\xi}  g_1(z) \dd m(z) = \sum_{k=0}^{m-1} \widehat g_1(k) \xi^k = : P_{m-1} g_1(\xi).
\end{align}

For $g= \Phi^* \1_{\{\lambda\}} \be' = \wt\Phi^* \1_{\{\lambda\}} \be' $ we have $g_1 (z) =  (1-\bar\lambda z)^{-1} (I_n-b(z) )  \bmu\{\la\}\be'$. So for this $g$, it follows from \eqref{Phi_m} that 
\begin{align}
\label{Phi_m 01}
\wt \Phi_m g (\xi) = \sum_{k=0}^{m-1}\widehat \f (k)\bmu\{\la\} \be' \xi^k  = P_{m-1}\f,\qquad\f(z) :=  (1-\bar\lambda z)^{-1} (I_n-b(z) ).
\end{align}

As discussed above, for $f = \1\ci{\{\lambda\}}\mbf{e}$, $\be\in\ran{\bmu\{\la\}}$ and $g
= \wt\Phi^* \1\ci{\{\lambda\}} \be' $, we have 
\begin{align*}
\lim_{m\to\infty} (\wt \Phi_m g, f)\ci{L^2(\bmu)} = (\wt \Phi g, f)\ci{L^2(\bmu)} = (\bmu\{\la\} \be',\be).
\end{align*}
Combining this with \eqref{Phi_m 01} we see that 
\begin{align*}
\sum_{k=0}^\infty (\xi^k \widehat \f (k) \bmu\{\la\}\be' ,\1\ci{\{\lambda\}} \be)\ci{L^2(\bmu)} 
= \sum_{k=0}^\infty \lambda^k(\widehat \f (k) \bmu\{\la\}\be' ,\bmu\{\la\} \be)\ci{\C^n}
= (\bmu\{\la\} \be', \be)\ci{\C^n}.    
\end{align*}
In other words, the Taylor series of the function $(\f (\fdot)\bmu\{\la\}\be', \bmu\{\la\}\be)\ci{\C^n} $ converges (at $z=\lambda$) to $(\bmu\{\la\} \be',\be)\ci{\C^n}$. Denoting $E=\ran \bmu\{\la\}$ we can see that this is equivalent to the convergence of the Taylor series of the function $P\ci E \f P\ci E$ to the Moore--Penrose inverse $\bmu\{\la\}^{[-1]}$ of $\bmu\{\la\}$.

The convergence of the Taylor series at $\lambda$ implies the radial convergence 
\begin{align*}
P\ci E \f(r\lambda) P\ci E \longrightarrow \left(\bmu\{\la\}\right)^{[-1]} \qquad \text{as } r\rightarrow 1^- .
\end{align*}

To show the non-tangential convergence we need the following simple and well-known fact, which can be proved using a standard normal families argument. 

\begin{prop}
\label{p:non-tangential}
Let $f$ be a bounded analytic function in a sector $S_\gamma := \{z\in\D: z\ne 0, |\arg z |<\gamma \}$, $0<\gamma\le\pi$. 

If there exists a ``radial limit''
\begin{align*}
\lim_{x\to 0+} f(x) = a, 
\end{align*}
then for any $0<\beta<\gamma$ the ``non-tangential'' limit as $z\to 0$ in $S_\beta$ exists and coincides with $a$, i.e.
\begin{align*}
\lim_{z\to 0: z\in S_\beta} f(z) =a. 
\end{align*}
\end{prop}

\begin{remark*}
The above ``radial limit'' can be taken along any ray in $S_\gamma$ originating from the origin. And, of course, the position and orientation of the sector is not important. 
\end{remark*}

\begin{remark}
\label{r:not dense range}
The condition that $f$ is bounded in $S_\gamma$ can be replaced by the assumption that the range of is not dense in $\C$, because in this case one can find a linear fractional transformation $\f$ such that $\f\circ f$ is bounded. 
\end{remark}

Let us now prove the non-tangential convergence of $P\ci E \f(z) P\ci E$ at $\lambda$. First of all, notice that $\re ( ( I_n - b(z)) \be, \be)\ci{\C^n} \ge 0$ for all $\be\in\C^n$. Second, notice that for a non-tangential approach regions $S=\{z\in \D : |\arg (z -{\lambda}) |<\gamma , \, |z-\lambda| <1\}$, $\gamma <\pi/2$, the values of $1/(1- \overline \lambda z)$ lie in a sector of aperture $2\gamma<\pi$. Therefore, for and $\be \in E$ the values of 
\begin{align}
\label{CAD 03}
\frac{( ( I_n - b(z)) \be, \be)\ci{\C^n} }{1- \overline \lambda z},   
\end{align}
$z\in S $,   lie in a sector with aperture less than $2\pi$, i.e.~the range is not dense in $\C$. Therefore (see Remark \ref{r:not dense range}), Proposition \ref{p:non-tangential} applies, and there is a non-tangential limit for any smaller approach region $S' = \{z\in \D : |\arg (z -{\lambda}) |<\beta , \, |z-\lambda| <1\} $, $\beta<\pi$. Since $\gamma<\pi/2$ can be arbitrary,   we get the non-tangential convergence of \eqref{CAD 03}
for any $\be\in E$ in the approach region of any aperture. 

Applying the polarization identity we get non-tangential convergence of \eqref{e-limitoperator}. 
This concludes the proof of Theorem \ref{t:CAD}.\hfill\qed

\subsection{An example}
\label{s:CAD counterexample}
It is a natural question to ask whether \emph{both} projections $P\ci E $ in \eqref{e-limit derivative} and \eqref{e-limitoperator} are really necessary. Below we will give an example of a function $b\in \mathscr{S}(2)$ such that the corresponding Clark measure $\bmu$ ($\bmu^\alpha$ with $\alpha = I_2$) has an atom at $1$, but the limits \eqref{e-limit derivative} and \eqref{e-limitoperator} (with $E=\ran \bmu\{\lambda\}$) fail to exist if one of the projections $P\ci E$ is missing. 

It is more convenient to work in the right half-plane $\C\ti r:=\{ z\in\C: \re z >0\}$. Let $\omega $, $\omega(z) = (1-z)/(1+z)$ be the standard conformal map from the unit disc $\D$ to $\C\ti r$. Note that $\omega^{-1} = \omega$, and $\omega (1) =0$. 
For an analytic function $f$ on $\D$ we denote by $\wt f$ its ``transplant'' to the  right half-plane $\C\ci r$, $\wt f := f\circ \omega^{-1} = f\circ \omega$ (recall that $\omega^{-1} =\omega$). 

Define the function $\wt H$ in $\C\ti r$ by
\begin{align}
\label{e: c-ex H}
\wt H(z) = \frac{1}{z} + \frac{1}{z^\gamma} + 1, \qquad z\in \C\ti r , 
\end{align}
where $0<\gamma <1$. 
Clearly $\wt H$ is Herglotz ($\re \wt H(z)\ge 0$), and, moreover on the imaginary axis
\begin{align*}
\re \wt H(ix) = |x|^{-\gamma} \cos(\gamma\pi/2) + 1. 
\end{align*}
Define
\begin{align*}
\wt \theta (z) :=   \frac{ \wt H(z) - 1}{ \wt H(z) + 1}.
\end{align*}
Clearly $|\wt\theta(z)|<1$ on $\C\ti R$, and, moreover, for small $x$ 
\begin{align}
\notag
|\wt\theta(ix)|^2 & =
\frac{(\im \wt  H(ix))^2 + (\re \wt  H(ix) -1)^2}{(\im \wt H(ix))^2 + (\re \wt H(ix) +1)^2} \\
\notag
&= 1 -\frac{4 \re \wt H(ix)}{(\im \wt H(ix))^2 + (\re \wt H(ix) +1)^2} \\
\label{e: theta< 01}
& \le 1- c |x|^{2-\gamma}. 
\end{align}
It is also easy to see that for any $\delta>0$ 
\begin{align}
\label{e: theta< 02}
|\wt\theta(ix)|^2 \le r(\delta) <1 
\end{align}
for all $|x|\ge \delta$. 

So, for  $0<\beta<1$ such that $2\beta> 2-\gamma$ and for sufficiently small $\e>0$ the matrix-valued function
\begin{align*}
\wt b(z) := 
\left(    \begin{array}{cc}
\wt\theta(z)  &  0\\ 
0 & 0
\end{array}  \right) 
+
\e \frac{z^\beta}{1+z} 
\left(    \begin{array}{cc}
0  &  1\\ 
1 &  1
\end{array}  \right) 
\end{align*}
is a strictly contractive one, i.e.~$\|\wt b(z)\|<1$ for $z\in\C\ti r$; by norm here we mean the operator norm.

Indeed, it is clear that $\wt b$ is a bounded analytic function in $\C\ti r$, so it is sufficient to show that $\|\wt b(ix)\|<1$ for all $x\ne 0$. We can estimate the bigger Frobenius (Hilbert--Schmidt) norm $\|\wit b(z)\|_2$. For small $x$ we get using \eqref{e: theta< 01} that
\begin{align}
\label{e: b < 01}
\| \wit b(z)\|_2^2 \le 1 - c|x|^{2-\gamma} + C\e^2   |x|^{2\beta}  .
\end{align}
Since $2\beta > 2-\gamma$ the right hand side is less than $1$ for sufficiently small $x$ and $\e$, i.e.~there exist $\e_0, \delta>0$ such that for all $|x|<\delta$ and $0<\e<\e_0$ the right hand side of \eqref{e: b < 01} is less than 1.  

For $|x|\ge \delta$ we use \eqref{e: theta< 02} and the fact that the function $(ix)^\beta /(1+ix)$ is bounded, to conclude that $\|\wit b(ix)\|_2^2 <1$ for sufficiently small  $\e$ (and for $|x|\ge\delta$). 

To see that the function $b:=\wit b \circ \omega$ gives us the desired example, we first notice 
that for an analytic function $f$ on $\D$ and $\wit f := f\circ \omega^{-1}= f\circ\omega$
\begin{align}
\label{e:lim 01}
\lim_{\xi\nt 1} (1-\xi) f(\xi) &= 2 \lim_{z\nt 0} z \wit f(z) \\
\intertext{and}
\label{e:lim 02}
\lim_{\xi\nt 1} \frac{f(\xi)}{(1-\xi)}  &= \frac12 \lim_{z\nt 0} \frac{\wit f(z)}{z} .
\end{align}
So, to verify the desired properties, we need to compute the limits of $(I_2-\wit b(z))/z$ and of its inverse as $z\nt 0$. 

It follows from the definition of $\wit \theta$ and \eqref{e: c-ex H} that for small $z$
\begin{align*}
1- \wit\theta (z) = \frac{2z}{1+ z^{1-\gamma} + 2z } = 2z\cdot (1+o(1)). 
\end{align*}

Then
\begin{align*}
I_2 - \wit b(z) & = 
\left(    \begin{array}{cc}
2z         &  -\e z^\beta \\ 
-\e z^\beta &  1 -\e z^\beta
\end{array}  \right) (1+o(1)) \\
&= 
\left(    \begin{array}{cc}
2z         &  -\e z^\beta \\ 
-\e z^\beta &  1 
\end{array}  \right) (1+o(1));
\intertext{the multiplication by $1+o(1)$ means that each entry of the matrix is multiplied by its own term $1+o(1)$.  Computing the inverse we get} 
( I_2 - \wit b(z) )^{-1} & = 
\frac{1}{2z}
\left(    \begin{array}{cc}
1         &  \e z^\beta \\ 
\e z^\beta &  2z
\end{array}  \right) (1+o(1)) \\
&= 
\frac12
\left(    \begin{array}{cc}
1/z         &  \e z^{\beta -1} \\ 
\e z^{\beta -1} &  2
\end{array}  \right) (1+o(1)). 
\end{align*}
Therefore, by Theorem \ref{t-Nevanlinna} we have for the Clark measure $\bmu$ (with $\alpha=I_2$) 
\begin{align*}
\bmu\{1\} = \lim_{\xi\nt 1} (1- \xi) \cdot (I_2 - b(\xi))^{-1} 
= 2 \lim_{z\nt 0} z \cdot (I_2-\wit b(z))^{-1} =
\left(    \begin{array}{cc}
1 &  0 \\ 
0 &  0
\end{array}  \right) ;
\end{align*}
here we used \eqref{e:lim 01} in the second equality. 

Similarly, we get using \eqref{e:lim 02} that 
\begin{align*}
\lim_{\xi\nt 1} (I_2-b(\xi))/(1-\xi) = \frac12 \lim_{z\nt 0} (I_2 - \wit b(z))/z =
\left(    \begin{array}{cc}
1 &  \infty \\ 
\infty &  \infty
\end{array}  \right) .
\end{align*}
This means that for $E=\ran \bmu\{1\}$ we have 
\begin{align*}
\lim_{\xi\nt 1} P\ci E \frac{(I_2-b(\xi))}{1-\xi} P\ci E= 
\left(    \begin{array}{cc}
1 &  0 \\ 
0 &  0
\end{array}  \right) , 
\end{align*}
but if we omit one of the projections  $P\ci E$, the (finite) limit does not exist. 
\hfill\qed

%%%%%%%%%%%%%%%%%%%%%%%%%%%%%%%%%%%%%%%%%%%%%%%%%%%
\section{Carath\'{e}odory angular derivatives and point masses}\label{s-CaraPoint}
%%%%%%%%%%%%%%%%%%%%%%%%%%%%%%%%%%%%%%%%%%%%%%%%%%%
We begin by presenting two results which compare the Carath\'{e}odory condition from Definition \ref{d-CCond} to the CAD, see Definition \ref{d-CAD}. 
In Theorems \ref{t-P} and \ref{t-CADimpliesCond}, we clarify the relation between the Carath\'{e}odory condition and the CAD.

Recall that the directional support of $\bmu^\alpha$ at $\lambda \in \mathbb{T}$ was introduced in \ref{d-direcionalsupport} to be ${\mathbf S}_\alpha (\lambda)= \left\{{\be}\in \C^n:\lim_{z\nt\lambda}b(z)^*\be  = \alpha^*\be \right\}.$
Now, consider the set
$${\mathbf E}_\alpha (\lambda):= \left\{\mbf{x}\in {\mathbf S}_\alpha (\lambda): b \text{ satisfies Carath\'{e}odory condition at }\lambda \text{ in codirection }\mbf{x}\right\}.$$
By Proposition \ref{p:CAD subspace}, ${\mathbf E}_\alpha (\lambda)$ is a subspace of $\C^n$.

\begin{thm}\label{t-P}
Function $b\alpha^*$ has 
CAD at $\lambda$ on the subspace $\mbf{E}_\alpha(\lambda)$, and 
\[
\left( \CAD\ci E (b\alpha^*)(\lambda)  \bx , \by \right)\ci{\C^n}  = \left\langle k^b_{\lambda, \bx}  , k^b_{\lambda, \by} \right\rangle_{\scr H(b)}.
\]
\end{thm}

\begin{proof}
Let $\mbf{x}\in \mbf{E}_\alpha(\lambda)$. Then by Proposition \ref{p-equivalences} the vector-valued function
\begin{align}\label{e-kernel}
k^b_{\lambda, \bx} (z) = \frac{\bx - b(z)\wt\bx}{1-z\overline\lambda}
\end{align}
with $\wt\bx = b^*(\lambda, \bx)$ belongs to $\scr H(b)$, and since  $\mbf{E}_\alpha(\lambda)\subset {\mathbf S}_\alpha (\lambda)$ we have
\begin{align}\label{e-yalpha}
\wt\by = b^*(\lambda, \bx) =\lim_{z\nt \lambda} b(z)^*\bx= \alpha^* \bx.
\end{align}
Now, take $\mbf{y}\in \mbf{E}_\alpha(\lambda)$. Then the non-tangential limit
\begin{align*}
\lim_{z\nt \lambda} \ipcn{\mbf{y}}{k^b_{\lambda, \bx} (z)}
\end{align*}
exists by Statement \cond3 of Proposition \ref{p-equivalences}. Using \eqref{e-kernel} and \eqref{e-yalpha}, we obtain the existence of
\begin{align*}
\lim_{z\nt \lambda} \ipcn{\mbf{y}}{\frac{I_n - b(z)\alpha^*}{1-z\overline\lambda}\bx }
\qquad \text{for all }\mbf{x}, \by\in \mbf{E}_\alpha(\lambda).
\end{align*}
This implies the existence of 
\begin{align*}
\lim_{z\nt \lambda} 
P\ci{\mbf{E}_\alpha(\lambda)}\frac{I_n - b(z)\alpha^*}{1-z\overline\lambda}
P\ci{\mbf{E}_\alpha(\lambda)},
\end{align*}
as was claimed.
\end{proof}

\begin{thm}\label{t-CADimpliesCond}
If $b\alpha^*$ has 
CAD on some subspace $E$ of $\C^n$, then for all $\mbf{x}\in E$, $P\ci{E}b\alpha^*|\ci{E}$ satisfies the Carath\'{e}odory condition in the codirection $\mbf{x}$.
\end{thm}

\begin{proof}
Without loss of generality assume $\alpha=I_n$. By the hypothesis, the non-tangential limit
\begin{align*}
\lim_{z\nt \lambda} 
P\ci{E}\frac{I_n - b(z)\alpha^*}{1-z\overline\lambda}
P\ci{E}
\end{align*}
exists. So, for all $\mbf{x}\in E$, the limit
\begin{align*}
\lim_{z\nt \lambda} 
P\ci{E}\frac{I_n - \alpha b(z)^*}{1-\overline z\lambda}
\mbf{x}
\end{align*}
exists, and so is the limit of the norm, 
\begin{align}
\label{e:lim norm}
\lim_{z\nt \lambda}  \frac{\| \bx - P\ci E\alpha b(z)^* \bx \|}{|1-\overline z\lambda|} <\infty. 
\end{align}

By the triangle inequality $\|\mbf u\|-\|\mbf v\| \le \| \mbf u - \mbf v \|$, so 
\[
\|\mbf{u}\|^2-\|\mbf{v}\|^2\le (\|\mbf{u}\|+\|\mbf{v}\|) \|\mbf{u}-\mbf{v}\| .
\]
Applying this to $\mbf{u} = \mbf{x}$ and $\mbf{v} = P\ci{E} \alpha b(z)^*\mbf{x}$, and using the fact that 
\[
\|P\ci{E} \alpha b(z)^*\mbf{x}\|\le \|\alpha b(z)^*\mbf{x}\| = \| b(z)^*\mbf{x}\| \le \|\bx\|, 
\]
we have
\begin{align*}
\|\mbf{x}\|^2 - \|P\ci{E} \alpha b(z)^*\mbf{x}\|^2 \le
2\|\mbf{x}\| \|\mbf{x} - P\ci{E} \alpha b(z)^*\mbf{x}\| .
\end{align*}
Since in a non-tangential approach region to $\lambda$ the quantities $1-|z|^2$ and $| 1 - \overline z \lambda |$ are compatible, 
we have 
\begin{align*}
\limsup_{z\nt \lambda} \frac{\|\mbf{x}\|^2 - \|P\ci{E} \alpha b(z)^*\mbf{x}\|^2} {1-|z|^2} 
& \le 2 \|\bx\| 
\limsup_{z\nt \lambda}
\frac{ \|\bx - P\ci{E} \alpha b(z)^*\mbf{x}\|} {  1- | z |^2 } \\
& \le C \|\bx\| 
\limsup_{z\nt \lambda}
\frac{ \|\bx - P\ci{E} \alpha b(z)^*\mbf{x}\|} { | 1-\overline z\lambda | } < \infty
\end{align*}
for some $0<C<\infty$; here, the boundedness of the final expression follows from \eqref{e:lim norm}.  

The resulting estimate 
\[
\limsup_{z\nt \lambda} \frac{\|\mbf{x}\|^2 - \| P\ci{E}\alpha b(z)^*\mbf{x}\|^2} {1-|z|^2} < \infty
\]
is clearly stronger than the Carath\'{e}odory condition \eqref{e:Caratheodory cond} ($\liminf\le\limsup$). 
\end{proof}

\begin{thm}\label{t-E}
We have $\ran{\bmu^\alpha\{\lambda\}} = \mbf{E}_\alpha(\lambda).$
\end{thm}

\begin{proof}
By Theorem \ref{t:CAD}, the CAD of $b\alpha^*$ at $\lambda$  exists on the subspace $\ran{\bmu^\alpha\{\lambda\}}$. Therefore,  by Theorem \ref{t-CADimpliesCond} the Carath\'{e}odory condition is satisfied in the codirection $\mbf{x}$ for all $\mbf{x}\in \ran{\bmu^\alpha\{\lambda\}}$. Thus, we obtain 
\begin{align*}
\ran{\bmu^\alpha\{\lambda\}}
\subset
\mbf{E}_\alpha(\lambda).
\end{align*}

Let us prove the reverse inclusion. Renaming $b\alpha^*$ by $b$ we can always assume $\alpha=I_n$ and skip the index $\alpha$.  

Define an operator $\Psi$ 
acting from $L^2(\bmu)$ to the space of analytic $\C^n$-valued functions on the unit disc $\D$, 
\begin{align}
\label{e:Phi^* 01}
\Psi f (z) := (I_n-b(z))  \cC[\bmu f](z)= (I_n-b(z)) \int_\T \frac{1}{1- z\overline  \xi} (\dd \bmu(\xi)) f(\xi);
\end{align}
here $\bmu$ is the Clark measure $\bmu^\alpha$ with $\alpha=I_n$ for $b$. 

Denote by $L^2_+(\bmu)$, the closure in $L^2(\bmu)$ of linear combinations of analytic in $\D$ vectorial rational fractions $r_{w, \bx}$
\begin{align*}
r_{w, \bx}(\xi) := \frac{\bx}{1-\overline w \xi}    , \qquad w\in\D, \ \bx\in \C^n.  
\end{align*}

Let $\bx\in \mbf{E}(\lambda)$, where recall $\mbf{E}(\lambda)=\mbf{E}_\alpha(\lambda)$ with $\alpha=I_n$. Then we have $\lim_{z\nt \lambda} b(z)^*\bx  = \mbf{x}$ and by Proposition \ref{p-equivalences}, the \emph{boundary reproducing kernel} $k_{\lambda,\bx}^b$ is given by 
\begin{align}
\label{e:k^b_la}
k_{\lambda,\bx}^b (z) = \frac{I_n - b(z)}{1 - z\overline \lambda} \bx .
\end{align}
By Lemma \ref{l:de Branges isometry} below, $\Psi: L^2_+(\bmu)\to \scr H(b)$ is a unitary operator. So, we see that
\begin{align*}
k_{\lambda,\bx}^b = \Psi f
\end{align*}
for some $f\in L^2_+(\bmu)$. 

Comparing the formula \eqref{e:Phi^* 01} for $\Psi$ with the formula \eqref{e:k^b_la} for $k_{\lambda,\bx}^b$ we can conclude that $f$ is supported at the point $\lambda$ only. More precisely $f = \1\ci{\{\lambda\}} \wt\bx$, where $\bmu\{\lambda\} \wt\bx =\bx$. But this exactly means that $\bx\in\ran \bmu\{\lambda\}$. 
\end{proof}

\begin{remark*}
If $b(0)=0$ the operator $\Psi$ used in this proof is exactly the first component of the adjoint Clark operator $\Phi^*$ from \cite{LTFinite}, which we also used in \eqref{Phi*}. For the general case (when $b(0)\neq 0$) it differs slightly from the one in \cite{LTFinite} as the measure there does not have the same normalization as ours.
\end{remark*}

\begin{lemma}
\label{l:de Branges isometry}
The operator $\Psi$, restricted to $L^2_+(\bmu)$ is a unitary operator between $L^2_+(\bmu)$ and the de Branges--Rovnyak space $\scr H (b)$. 
\end{lemma}

\begin{proof}
Define 
\begin{align*}
\wit k_w(\xi) := \frac{I_n - b(w)^*}{1-\overline{w}\xi}, \qquad w\in\D, \ \xi\in\T. 
\end{align*}
We will show that for any $\bx\in\C^n$ there holds 
\begin{align}
\label{e:Phi^* k_w}
\Psi \wit k_w\bx  &= k^b_w \bx  = \frac{I_n- b(\xi)b(w)^*}{1- w\overline\xi}\bx ,
\intertext{where, recall,  $k_w^b$ is the reproducing kernel for the de Branges--Rovnyak space $\scr H(b)$, and}
\label{e:Phi^* isometry}
\left\langle \wit k_w\bx , \wit k_z\by \right\rangle_{L^2(\bmu)} & =    \left(   \frac{I_n- b(z)b(w)^*}{1- w\overline z}\bx, \by          \right)_{\C^n}
=  \left\langle  k_w^b\bx ,  k_z^b\by \right\rangle_{\scr H(b)}.
\end{align}
That means $\Psi$ is an isometry from a dense set in $L^2_+(\bmu)$ (linear combinations of functions $\wt k_w \bx$) to $\scr H(b)$. The subspace $\Psi L^2_+(\bmu)$ contains linear combinations of functions  $k^b_w \bx$, $w\in\D$, $\bx\in\C^n$, and such linear combinations are dense in $\scr H(b)$. Indeed, if $f\in\scr H(b) $ is orthogonal to all $k^b_w\bx$, then 
\[
\left(f(w), \bx\right)\ci{\C^n} = \langle f, k_w^b\bx\rangle\ci{\scr H(b)} = 0, 
\]
for all $w\in \D$ and all $\mbf{x}\in \C^n$. So we have $f\equiv 0$. 

So, $\Psi :L^2_+(\bmu)\to\scr H(b)$ is indeed a unitary operator.

To prove \eqref{e:Phi^* k_w} let us compute $\Psi \wt k_w$ by computing $\Psi \wt k_w \be_k$, where $(\be_k)_{k=1}^n$ is the standard basis in $\C^n$. 
Using the formula
\begin{align*}
\frac{1}{1-\overline{w}\xi} \cdot \frac{1}{1- z\overline\xi} = 
\frac{1}{1-\overline{w} z } \cdot\frac12\cdot
\left( \frac{1 + \overline{w}\xi}{1-\overline{w}\xi} + \frac{1 + z\overline\xi}{1- z\overline\xi} 
\right)
\, , \ 
\end{align*}
where $w, z\in\D$, $\xi\in\T$, 
we can see that
\begin{align}
\label{e:Phi^* k_w 02}
\notag
[\Psi \wit k_w  ](z) & =  (I_n-b(z)) \left( \int_\T   \frac{1}{1-\overline{w}\xi} \cdot \frac{1}{1- z\overline\xi} \dd\bmu(\xi) \right) (I_n- b(w)^*) \\
&=  \frac12 \cdot\frac{I_n- b(z)}{1-\overline{w} z } 
\left( \cC_2 \bmu(w)^* + \cC_2\bmu(z) 
\right)  (I_n - b(w)^*) ,
\end{align}
where $\cC_2$ is the 
Herglotz transform, 
\begin{align*}
\cC_2\nu (z):= \int_\T \,\frac{1 + z\overline \xi}{1 - z\overline \xi} \, \dd\nu(\xi).
\end{align*}
Using \eqref{e-Herglotz} we see that 
\[
\cC_2\bmu (z) = \frac{I_n + b(z)}{I_n- b(z)} - i\im  \frac{I_n + b(0)}{I_n- b(0)}\,, 
\]
(we write the product $(I_n + b(z))(I_n - b(z))^{-1}$ as fraction to emphasize that terms commute),  so
\[
\cC_2 \bmu(w)^* + \cC_2\bmu(z) =   \frac{I_n + b(w)^*}{I_n- b(w)^*} +     \frac{I_n + b(z)}{I_n- b(z)} \, .
\]
Therefore
\[
(I_n- b(z)) \left( \cC_2 \bmu(w)^* + \cC_2\bmu(z)   \right) (I_n - b(w)^*) =
2\cdot (I_n - b(z) b(w)^* ), 
\]
and \eqref{e:Phi^* k_w 02} gives us 
\begin{align}
\label{e:Phi^* k_w 03}
[\Psi \wit k_w  ](z) = \frac{I_n - b(z)b(w)^*}{1 - z\overline w}  = k^b_w , 
\end{align}
so \eqref{e:Phi^* k_w} is proved. 

To prove \eqref{e:Phi^* isometry} let us notice that 
\[
\left\langle \wt k_w\bx , \wt k_z \by   \right\rangle_{L^2(\bmu)}  = \left( [\Psi \wt k_w] (z) \bx, \by  \right)_{\C^n} .
\]
But $\Psi \wt k_w$ is already computed, see \eqref{e:Phi^* k_w 03}, so \eqref{e:Phi^* isometry} follows immediately. 
\end{proof}

\begin{remark*}
In the above proof it is essential that both $z, w\in \D$. 
\end{remark*}

\end{document}